\documentclass[12pt,leqno]{article}
\usepackage{amsfonts}
\usepackage{amsmath, amsthm, amsfonts, amssymb, color}
\usepackage{mathrsfs}

\usepackage{chngcntr}
\usepackage{graphicx} 
\usepackage{float}
\usepackage{diagbox}

\usepackage[
singlelinecheck=false 
]{caption}

\usepackage{siunitx}
\renewcommand{\bar}{\overline}
\renewcommand{\hat}{\widehat}
\renewcommand{\tilde}{\widetilde}

\newtheorem{thm}{Theorem}[section]
\newtheorem{cor}[thm]{Corollary}

\newtheorem{lem}[thm]{Lemma}

\newtheorem{exa}[thm]{Example}
\newtheorem{assu}[thm]{Assumption}
\newtheorem{com}[thm]{Comment}
\theoremstyle{definition}
\newtheorem{defn}{Definition}[section]
\newcommand{\scr}[1]{\mathscr #1}
\definecolor{wco}{rgb}{0.5,0.2,0.3}
\allowdisplaybreaks
\numberwithin{equation}{section} \theoremstyle{remark}
\newtheorem{rem}{Remark}[section]

\newcommand{\ua}{\uparrow}

\title{{\bf Explicit Numerical Approximations for SDDEs in Finite and Infinite Horizons using the Adaptive EM Method: Strong Convergence and Almost Sure Exponential Stability }
}
\author{
{\bf  Ulises Botija-Munoz and Chenggui Yuan
 }\\
\footnotesize{Department of Mathematics, Swansea University, Bay Campus, SA1 8EN, UK}\\
}

\begin{document}
\def\A{\mathscr{A}}
\def\G{\mathscr{G}}
\def\eq{\equation}
\def\bg{\begin}
\def\ep{\epsilon}
\def\x{\|x\|}
\def\y{\|y\|}
\def\xr{\|x\|_r}
\def\xrr{(\sum_{i=1}^T|x_i|^r)^{\frac{1}{r}}}
\def\R{\mathbb R}
\def\ff{\frac}
\def\ss{\sqrt}
\def\B{\mathbf B}
\def\N{\mathbb N}
\def\kk{\kappa} \def\m{{\bf m}}
\def\dd{\delta} \def\DD{\Dd} \def\vv{\varepsilon} \def\rr{\rho}
\def\<{\langle} \def\>{\rangle} \def\GG{\Gamma} \def\gg{\gamma}
  \def\nn{\nabla} \def\pp{\partial} \def\EE{\scr E}
\def\d{\text{\rm{d}}} \def\bb{\beta} \def\aa{\alpha} \def\D{\scr D}
  \def\si{\sigma} \def\ess{\text{\rm{ess}}}\def\lam{\lambda}
\def\beg{\begin} \def\beq{\begin{equation}}  \def\F{\scr F} \def\s{\underline s} \def\t{\underline t}  
\def\Ric{\text{\rm{Ric}}} \def \Hess{\text{\rm{Hess}}}
\def\e{\text{\rm{e}}} \def\ua{\underline a} \def\OO{\Omega}  \def\oo{\omega}
 \def\tt{\tilde} \def\Ric{\text{\rm{Ric}}}
\def\cut{\text{\rm{cut}}} \def\P{\mathbb P} \def\ifn{I_n(f^{\bigotimes n})}
\def\C{\scr C}      \def\alphaa{\mathbf{r}}     \def\r{r}
\def\gap{\text{\rm{gap}}} \def\prr{\pi_{{\bf m},\varrho}}  \def\r{\mathbf r}
\def\Z{\mathbb Z} \def\vrr{\varrho} \def\l{\lambda}
\def\L{\scr L}\def\Tilde{\tilde} \def\TILDE{\tilde}\def\II{\mathbb I}
\def\i{{\rm in}}\def\Sect{{\rm Sect}}\def\E{\mathbb E} \def\H{\mathbb H}
\def\M{\scr M}\def\Q{\mathbb Q} \def\texto{\text{o}} \def\LL{\Lambda}
\def\Rank{{\rm Rank}} \def\B{\scr B} \def\i{{\rm i}} \def\HR{Hat{\R}^d}
\def\to{\rightarrow}\def\l{\ell}\def\ll{\lambda}
\def\8{\infty}\def\ee{\epsilon} \def\Y{\mathbb{Y}} \def\lf{\lfloor}
\def\rf{\rfloor}\def\3{\triangle}\def\H{\mathbb{H}}\def\S{\mathbb{S}}\def\1{\lesssim}
\def\va{\varphi}

\def\R{\mathbb R}  \def\B{\mathbf
B}
\def\N{\mathbb N} \def\kk{\kappa} \def\m{{\bf m}}
\def\dd{\delta} \def\DD{\Delta} \def\vv{\varepsilon} \def\rr{\rho}
\def\<{\langle} \def\>{\rangle} \def\GG{\Gamma} \def\gg{\gamma}
  \def\nn{\nabla} \def\pp{\partial} \def\EE{\scr E}
\def\d{\text{\rm{d}}} \def\bb{\beta} \def\aa{\alpha} \def\D{\scr D}
  \def\si{\sigma} \def\ess{\text{\rm{ess}}}
\def\beg{\begin} \def\beq{\begin{equation}}  \def\F{\mathcal F}
\def\Ric{\text{\rm{Ric}}} \def\Hess{\text{\rm{Hess}}}
\def\e{\text{\rm{e}}} \def\ua{\underline a} \def\OO{\Omega}  \def\oo{\omega}
 \def\tt{\tilde} \def\Ric{\text{\rm{Ric}}}
\def\cut{\text{\rm{cut}}} \def\P{\mathbb P} \def\ifn{I_n(f^{\bigotimes n})}
\def\C{\scr C}      \def\aaa{\mathbf{r}}     \def\r{r}
\def\gap{\text{\rm{gap}}} \def\prr{\pi_{{\bf m},\varrho}}  \def\r{\mathbf r}
\def\Z{\mathbb Z} \def\vrr{\varrho} \def\ll{\lambda}
\def\L{\scr L}\def\Tt{\tt} \def\TT{\tt}\def\II{\mathbb I}
\def\i{{\rm in}}\def\Sect{{\rm Sect}}\def\E{\mathbb E} \def\H{\mathbb H}
\def\M{\scr M}\def\Q{\mathbb Q} \def\texto{\text{o}} \def\LL{\Lambda}
\def\Rank{{\rm Rank}} \def\B{\scr B} \def\i{{\rm i}} \def\HR{\hat{\R}^d}
\def\to{\rightarrow}\def\l{\ell}
\def\8{\infty}\def\X{\mathbb{X}}\def\3{\triangle}
\def\V{\mathbb{V}}\def\M{\mathbb{M}}\def\W{\mathbb{W}}\def\Y{\mathbb{Y}}\def\1{\lesssim}
\def\eq{\eqref}

\def\va{\varphi}
\def\l{\lambda}
\def\var{\varphi}

\def\time{0 \leq t \leq T}

\renewcommand{\bar}{\overline}
\renewcommand{\hat}{\widehat}
\renewcommand{\tilde}{\widetilde}
  
\maketitle
\begin{abstract}
In this paper we investigate explicit numerical approximations for stochastic differential delay equations (SDDEs) under a local Lipschitz condition by employing the adaptive Euler-Maruyama (EM)  method. Working in both finite and infinite horizons, we achieve strong convergence results by showing the boundedness of the $p$th moments of the adaptive EM solution. We also obtain the order of convergence in finite horizon. In addition, we show almost sure exponential stability of the adaptive approximate solution for both SDEs and SDDEs.   
\end{abstract}


\section{Introduction}\label{sec1}

Consider the following SDDEs
\begin{equation} \label{SDDE counterexample}
	dY_t = (-2Y_t - Y^3_t+ \ff 1 2 Y_t \sin(Y_{t-1}))dt + \sqrt 2 Y_t \cos(Y_{t-1}) dW_t
\end{equation}
with initial data $\xi \in C([-1,0]; \mathbb R),\xi(0) = c \in \mathbb R / \{0\}.$
Using  \cite[Theorem 1]{Wu}, we can show that the exact solution of the SDDE \eqref{SDDE counterexample} is almost sure exponentially stable, i.e.
$$ \limsup_{t \rightarrow \infty} \ff 1 t \log|Y_t| \leq -\lambda \text{ a.s., } \quad \lambda >0.           $$
However, the discrete (standard) EM approximate solution
\begin{equation}\label{EM counterexample'}
\begin{cases} 
	X_k &= \xi(k \Delta) \quad k = -m, -m+1, ..., 0, \\
	X_{k+1} &= X_k-X_k[(2 +X_k^2  - \frac{1}{2} X_k  \sin(X_{k-1}))\Delta+ \sqrt{2} \cos(X_{k-1}) \Delta W_k], \quad k = 0,1, \ldots
\end{cases}
\end{equation}
where $\Delta = 1 / m, m \in \mathbb N,$
is not almost sure exponentially stable. This means that it does not exist a constant $\eta>0$ and a $\Delta^{*} \in (0,1)$ such that for all $\Delta \in (0, \Delta^{*})$ 
$$ \limsup_{k \rightarrow \infty} \ff 1 {k \Delta} \log|X_k| \leq -\eta  \ \text{ a.s. }.           $$

On the contrary, as we will see in Section \ref{stab SDDEs}, the adaptive EM approximate solution to equation \eqref{SDDE counterexample} is almost sure exponentially stable. \\

The classical existence-and-uniqueness theorem for SDDEs requires the drift and diffusion functions to satisfy a local Lipschitz condition and a linear growth condition (see \cite{Mao2}). However, in applications there are many SDDEs which do not satisfy the linear growth condition. The Khasminskii-type theorem in \cite{Mao3} enables to prove existence-and-uniquess for a class of SDDEs using a weaker condition than the linear growth one. Thus it is desirable, under these weaker conditions, to find numerical approximate solutions that converge strongly to the exact solution. In 2003, Mao \cite{Delay} proved strong convergence using the EM scheme and assuming the boundedness of the $p$th moments for both the exact and the numerical solution. It is well-known that the linear growth condition implies the boundedness of the $p$th moments for the EM approximate solution. But when the drift function grows faster than linear, the standard EM scheme fails, see the example with polynomial growth in Hunter \cite{Hutz}. Therefore, modifications of the EM scheme for SDDEs that provide explicit approximate solutions have appeared in the last few years to account for this issue. Examples of these are the tamed \cite{Tamed} and the truncated \cite{Guo} methods.

In 2020, Wei and Giles \cite{Giles} obtained strong convergence for the numerical solution of a SDEs in a finite horizon under local Lipschitz and one-sided linear growth conditions. They use an adaptive EM scheme in which the time step is not a constant, but a function of the solution at that point in time. They also, under more restrictive conditions, showed strong convergence in infinite horizon. Here, we extend their work to SDDEs in both, finite and infinite horizons.

Additionally, to study the stability of numerical solutions is an important topic. Moment stability for SDDEs has been studied extensively, see for example \cite{Baker, Mao4}. Almost sure (a.s.) exponential stability is usually derived from moment stability by means of the Borel-Cantelli lemma and Markov's inequality (see \cite{Higham3}). In Wu et al. \cite{Wu}, using the EM and the Backward EM (BEM) methods, a.s. exponential stability was studied for SDDEs without using moment stability. Their approach was based on the martingale convergence theorem. They required the linear growth condition when dealing with the standard EM scheme. When they weaken the linear growth condition to the one-sided linear growth condition for the diffusion function, they showed how the standard EM approximate solution loses the stability of the exact solution. Then they showed that under the one-sided linear growth condition, the a.s. exponential stability can be achieved by using the BEM method. This method is implicit and therefore more computationally expensive than explicit methods like the adaptive EM. Here, we obtained a.s. exponential stability using the adaptive method. 

The rest of the paper is structured as follows. Section \ref{pre} introduces some preliminary notation and the type of SDDE we will work with in the rest of the paper. Section \ref{adaptive EM} describes the adaptive EM method. Section \ref{finite} deals with strong convergence and order of convergence in finite horizon. In Section \ref{infinite} we obtained the boundedness of the $p$th moments for the adaptive EM approximate solution in infinite horizon. In Section \ref{stab SDDEs} we show that almost surely  exponential stability of the adaptive EM solution for SDDEs can be recovered  and provide illustration for  counterexample \ref{SDDE counterexample}. In Section \ref{simulations}, we present some simulations to illustrate the results in Section \ref{stab SDDEs}.

\section{Preliminaries} \label{pre}

Throughout this paper, let $(\Omega, \F, \P)$ be a filtered complete probability space where the filtration $\{ \F_t \}_{t \geq 0}$ satisfies the usual conditions (i.e. it is right continuous and $\F_0$ contains all $\P$-null sets). Let $\tau > 0$ and $T>0$ be constants and denote $C([-\tau,0];\R^m)$ the space of all continuous functions from $[-\tau,0]$ to $\R^m$ with the norm $||\phi|| = \sup_{-\tau \leq \theta \leq 0}|\phi(\theta)|$. Let $\{W_t\}_{\time}$ be a standard $d$-dimensional Brownian motion. For a $\R^m$-vector $v$, we denote the Euclidean norm by $|v|:=(|v_1|^2+...+|v_m|^2)^\ff 1 2$ and the inner product of two $\R^m$-vectors $v$ and $w$ by $\<v,w\>:=v_1 w_1 + ... + v_m w_m.$ For a $m \times d$ matrix A, we denote the Frobenius matrix norm by $||A||:=\sqrt{ \text{trace}(A^T A)}.$ \\
Consider an $m$-dimensional SDDE of the form
\begin{equation} \label{SDE}
	dY_t = f(Y_t, Y_{t - \tau})dt + g(Y_t,Y_{t-\tau})dW_t
\end{equation}
on $t \geq 0,$ where $f:\R^m \times \R^m \rightarrow \R^m$ and $g:\R^m \times \R^m \rightarrow \R^{m \times d}$ are Borel-measurable functions, and the initial data satisfies the following condition: for any $p \geq 2$
$$\{Y(\theta):-\tau \leq \theta \leq 0\}=\xi \in L^p_{\F_0}([-\tau,0];\R^m),$$
that is $\xi$ is a $\F_0$-measurable $C([-\tau,0];\R^m)$-valued random variable such that $E||\xi||^p < \infty.$

\section{Adaptive Method} \label{adaptive EM}
The time step is determined by a function $h^\delta:\R^m \rightarrow \R^+$ with $\delta \in (0,1).$ The family of functions $\{h^\delta\}_{0 < \delta < 1}$ is not specifically defined, it just has to satisfy certain conditions that we will describe later in the next assumption. To see concrete examples where the function $h^\delta$ is fully specified, see Section 4 in \cite{Giles} or the example at the end of this paper.
We now define the adaptive method for SDDEs. Set 
$$\hat X_0:=\xi(0),\ h_0^\delta:=h^\delta(\hat X_0), \ t_1:=h_0^\delta.$$
We introduce the continuous-time step (auxiliary) process $\bar X$. Define
$$\bar X_t := \xi(t), t \in [-\tau,0), \quad \bar X_t := \xi(0),  t \in [0,t_1).$$ For $t_1$ we define the discrete-time approximate solution $\hat X$ as
\begin{align*}
\hat X_{t_1}&:= \hat X_0 + f(\bar X_0,\bar X_{-\tau})h^\delta_0 + g(\bar X_0,\bar X_{-\tau})\Delta W_0, \\ 
h^\delta_1&:=h^\delta(\hat X_{t_1}), \ t_2=t_1+h^\delta_1, \\ 
\bar X_t&:=\hat X_{t_1}, t \in [t_1,t_2),
\end{align*}
where $\Delta W_0:=W_{t_{1}}-W_{0}.$ Then for a generic $t_n$ we define  
\begin{align}
\hat X_{t_{n+1}}&:= \hat X_{t_n} + f(\bar X_{t_n},\bar X_{t_n-\tau})h^\delta_n + g(\bar X_{t_n},\bar X_{t_n-\tau})\Delta W_n, \label{dis-scheme} \\ 
h^\delta_{n+1}&:=h^\delta(\hat X_{t_{n+1}}), \ t_{n+2}=t_{n+1}+h^\delta_{n+1}, \nonumber \\ 
\bar X_t&:=\hat X_{t_{n+1}}, t \in [t_{n+1},t_{n+2}), \nonumber
\end{align}
where $\Delta W_n:=W_{t_{n+1}}-W_{t_n}.$ For every path $\omega \in \Omega$, we continue the recursion (3.1) until $n=N(\omega):= \inf \{n \in \mathbb Z^+: t_n(\omega) \geq T\}.$ Note that $t_n$ and $h^\delta_n$ are random variables. For every $\omega$, let $r=r(\omega)$ be such $t_r \leq \tau \leq  t_{r+1}$. Then we define the continuous-time step (auxiliary) process $\tilde X$ as
\begin{equation}\label{tilde-process}
\tilde X_t := \bar X_{-\tau}, t \in [-\tau,t_1-\tau), \quad \tilde X_t:=\bar X_{t_1-\tau}, t \in [t_1-\tau,t_2-\tau),...., \quad \tilde X_t := \bar X_{t_{r}-\tau},  t \in [t_{r}-\tau,t_{r+1}-\tau), $$
$$\tilde X_t := \bar X_{t_{r+1}-\tau},  t \in [t_{r+1}-\tau,t_{r+2}-\tau), \quad \tilde X_t := \bar X_{t_{r+n}-\tau},  t \in [t_{r+n}-\tau,t_{r+n+1}-\tau)
\end{equation}
for $n=1,...,N-r.$ We now define the the continuous approximate solution
\begin{align} 
	X_t&:= \xi(t),  \quad t \in [-\tau,0];  \nonumber \\
	X_t&:= X_0 + \int_0^t f(\bar X_s, \tilde X_{s-\tau})ds + \int_0^t g(\bar X_s, \tilde X_{s-\tau})dW_s, \quad  t > 0. \label{cont-time approx sol}
\end{align}
Note that $\hat X_{t_n}=\bar X_{t_n}=X_{t_n}$ for $n=0,1,...,N.$
\color{black}

\section{Convergence of the numerical solutions on finite time interval} \label{finite}
In this section we will work on a finite time interval $[-\tau, T], T>0,$ and investigate the convergence of the numerical solutions to the exact solution on $[0,T]$. 

\begin{assu} \label{assu f and g} The functions $f$ and $g$ satisfy the local Lipschitz condition: for every $R>0$ there exists a positive constant $C_R$ such that
	\begin{equation} \label{LL}
		| f(x,y)-f(\bar x, \bar y)| +||g(x,y)-g(\bar x,\bar y)|| \leq C_R (|x- \bar x|+|y- \bar y|)
	\end{equation}
	for all $x,y,\bar x, \bar y \in \R^m$ with $|x|\lor|y|\lor|\bar x|\lor |\bar y| \leq R.$ Furthermore, there exist two constants $\alpha, \beta \geq 0$ such that for all $x,y \in \R^m,$ $f$ satisfies the one-sided linear growth condition: 
	\begin{equation} \label{f one-sided LG}
		\<x,f(x,y)\> \leq \alpha(|x|^2+|y|^2)+ \beta
	\end{equation}
	and $g$ satisfies the linear growth condition: 
	\begin{equation} \label{g LG}
		||g(x,y)||^2 \leq \alpha(|x|^2+|y|^2)+ \beta.
	\end{equation}	
\end{assu}

\begin{assu} \label{assu h}
	The time step function $h^\delta:\R^m \rightarrow \R^+, \ \delta \in (0,1),$ is continuous, strictly positive and bounded by $\delta T,$ i.e.
	\begin{equation}\label{h2}
		0 < h^\delta(x) \leq \delta T  \quad \text{ for all } x \in \R^m.
	\end{equation}
	
	Furthermore, there exist constants $\alpha, \beta > 0$ such that for all $x,y \in \R^m.$
	\begin{equation}\label{h1}
		\<x,f(x,y)\> + \ff 1 2 h^\delta(x)|f(x,y)|^2 \leq \alpha(|x|^2+|y|^2)+ \beta.
	\end{equation}
\end{assu}
Note that condition \eqref{h1} implies condition \eqref{f one-sided LG} with the same values of $\alpha$ and $\beta$.   

\subsection{The boundedness of the $p$th moments of the exact solution and the numerical solutions}

\subsubsection{Exact solution}

In this subsection we will discuss the $p$th moments of the exact solution to SDDE \eqref{SDE}.
\begin{lem}
	If the SDDE \eqref{SDE} satisfies Assumption \eqref{assu f and g}, then there exists a positive constant $C$ such that for any $p \geq 2$ 
	\begin{equation}
		\E \left[ \sup_{\time} |Y_t|^p \right] \leq C.
	\end{equation} 
\end{lem}

\begin{proof}
	The proof is given in Lemma 3.2 in \cite{Tamed}.
\end{proof}

\subsubsection{Adaptive EM numerical solutions} 

In this subsection, the $p$th moments of numerical solution will be investigated. In the standard Euler-Maruyama method the discretisation times $\{t_n \}$ are built using a constant time step $\Delta$ and a fixed number of steps $N \in \N,$ i.e. $t_N = N \Delta = T.$ However, in the adaptive method, $\{t_n \}$  is a sequence of random variables and there is no guarantee that it reaches $T$ in a finite number of steps. Thus, we have the following definition.

\begin{defn}
	We say that the time horizon $T$ is attainable if $\{t_n\}$ reaches $T$ in a finite number of steps $N$, i.e. for almost all $\omega \in \Omega,$ there exists a $N(\omega)$ such that $t_{N(\omega)} = \sum_{n=0}^{N(\omega)} h^{\delta} (X_{t_n}) \geq T.$ 
\end{defn}  

\begin{thm} \label{thm - boundedness pth moments} 
	If the SDDE \eqref{SDE} and the function $h^\delta$ satisfy Assumption \ref{assu f and g} and \ref{assu h} respectively, then $T$ is attainable and for all $p>0$ there exists	a constant $C>0$ dependent on $T$ and $p,$ but independent of $h^\delta_n,$ such that
	\begin{equation}
		\E \left[ \sup_{\time} |X_t|^p \right] \leq C.
	\end{equation} 
\end{thm}

The discrete-time approximate solution defined in \eqref{dis-scheme} need not be bounded. In order to show that $T$ is attainable and prove Theorem \ref{thm - boundedness pth moments}, we need to work with a bounded approximate solution. To this end we now introduce the following auxiliary scheme.
Let $K > ||\xi||.$ Set $\hat X^K_0:=\xi(0), h_0^\delta :=h^\delta(\hat X_0), t_1:=h_0^\delta$ and $\bar X^K_t:=\xi(t),t \in [-\tau,0), \bar X_t^K:=\xi(0), t \in[0,t_1)$. Consider the function $\Phi_K:\R^m \rightarrow \R^m, \Phi(x)=\min(1,K/|x|)x$.  Then for every $\omega \in \Omega$ and for $n=0,1,...N(\omega)$ (where $N(\omega)$ is defined below), we define
\begin{equation}\label{K-scheme}
\begin{aligned}
\hat X^K_{t_{n+1}}&:=\Phi_K(\hat X_{t_n}^K + f(\hat X_{t_n}^K, \bar X^K_{t_n-\tau})h_n^\delta +g(\hat X_{t_n}^K, \bar X^K_{t_n-\tau})\Delta W_n) \\
h^\delta_{n+1}&:=h^\delta(X_{t_{n+1}}^K), t_{n+2}:=t_{n+1}+h^\delta_{n+1}, \\
\bar X^K_t &:= \hat X_{t_{n+1}}^K, t \in [t_{n+1},t_{n+2}).
\end{aligned}
\end{equation}
Define for $n=0,...,N-r$
\begin{equation}\label{tilde-process-K}
\tilde X^K_t:=\bar X^K_{t_n-\tau}, t \in [t_n-\tau,t_{n+1}-\tau),
\end{equation}
where $r=r(\omega)$ is such that $t_r \leq \tau \leq  t_{r+1}.$ We now define the the continuous approximate solution
\begin{align} 
X^K_t&:= \xi(t),  \quad t \in [-\tau,0];  \nonumber \\
X^K_t&:= \Phi_K\left(\hat X_{\t}^K + f(\hat X_{\t}^K, \bar X^K_{\t-\tau})(t-\t) +g(\hat X_{\t}^K, \bar X^K_{\t-\tau})(W_t - W_{\t})\right)  \label{cont-time approx sol-K} \quad  t > 0, 
\end{align}
where $\underline{t}:=\max\{t_n: t_n \leq t\}.$ Note that $X^K_{t_n}=\hat X^K_{t_n}=\bar X^K_{t_n}.$ 
\begin{lem} \label{lem - aux boundedness of the pth moments}
	Let $p \geq 4,$ the SDDE \eqref{SDE} satisfy Assumption \ref{assu f and g} and the function $h^\delta$ satisfy Assumption \ref{assu h}. Then, for the auxiliary scheme defined by \eqref{cont-time approx sol-K}, $T$ is attainable and for all $p \geq 4$ there exists a constant $C$ dependent on $T$ and $p,$ but independent of $h^\delta_n$ and $K$ such that
	\begin{equation}
		\E \left[ \sup_{\time} | X^K_t|^p \right] \leq C.
	\end{equation} 
\end{lem}

\begin{proof}
Fix $\delta \in  (0,1).$ Since $h^\delta$ is continuous and strictly positive, $\inf_{|x| \leq K} h^\delta(x) >0.$ This implies that for every $\omega \in \Omega$
$$ \liminf_{n \rightarrow \infty} h_n^\delta(\omega) = \liminf_{n \rightarrow \infty} h^\delta(\hat X^K_{t_n}(\omega))  > 0,$$
so $\lim_{n \rightarrow \infty} t_n(\omega) = \sum_{n=0}^{\infty} h_n^\delta(\omega)= \infty$ for all $\omega \in \Omega$ and $T$ is attainable in the bounded scheme. \\
Now we will prove the boundedness of the $p$th moments and the upper bound will be independent of $h^\delta_n$ and $K$. Let $t \in [0,T]$. Define $\underline{t}:=\max\{t_n: t_n \leq t\},$ \ and $n_t:=\max\{n:t_n \leq t\}.$ 
Using \eqref{K-scheme} and since for any $x \in \R^m$, $|\Phi(x)|^2 \leq |x|^2$, we have that for $n = 0$ to $n = n_t -1,$ 
\begin{align*} 
	|\hat X^K_{t_{n+1}}|^2 &\leq  |\hat X^K_{t_n} + f(\hat X^K_{t_n},\bar X^K_{t_n-\tau})h_n + g(\hat X^K_{t_n},\bar X^K_{t_n-\tau})\Delta W_n|^2    \nonumber \\
	&= \<\hat X^K_{t_n},\hat X^K_{t_n}\> + 2 \<\hat X^K_{t_n}, f(\hat X^K_{t_n},\bar X^K_{t_n-\tau})h_n \> + \<f(\hat X^K_{t_n},\bar X^K_{t_n-\tau})h_n, f(\hat X^K_{t_n},\bar X^K_{t_n-\tau})h_n \> \nonumber \\  
	&+ 2\<\hat X^K_{t_n} + f(\hat X^K_{t_n},\bar X^K_{t_n-\tau})h_n, g(\hat X^K_{t_n},\bar X^K_{t_n-\tau})\Delta W_n\> \\
	&+ \<g(\hat X^K_{t_n},\bar X^K_{t_n-\tau})\Delta W_n, g(\hat X^K_{t_n},\bar X^K_{t_n-\tau})\Delta W_n\> \nonumber \\
	&= |\hat X^K_{t_n}|^2 + 2 h_n \left[\<\hat X^K_{t_n},f(\hat X^K_{t_n},\bar X^K_{t_n-\tau})\>+ \ff 1 2 h_n |f(\hat X^K_{t_n},\bar X^K_{t_n-\tau})|^2 \right] \nonumber \\
	&+2 \<\hat X^K_{t_n}+f(\hat X^K_{t_n},\bar X^K_{t_n-\tau})h_n, g(\hat X^K_{t_n},\bar X^K_{t_n-\tau})\Delta W_n \>+|g(\hat X^K_{t_n},\bar X^K_{t_n-\tau})\Delta W_n|^2 \nonumber \\
	&\leq |\hat X^K_{t_n}|^2 + 2 h_n \alpha(|\hat X^K_{t_n}|^2 + |\bar X^K_{t_n-\tau}|^2 )+  2 h_n \beta \nonumber \\
	&+2  \<\hat X^K_{t_n}+f(\hat X^K_{t_n},\bar X^K_{t_n-\tau})h_n,g(\hat X^K_{t_n},\bar X^K_{t_n-\tau})\Delta W_n \>+|g(\hat X^K_{t_n},\bar X^K_{t_n-\tau})\Delta W_n|^2, 
\end{align*}
where in the last step we have used condition \eqref{h1}. Note that, since it is irrelevant in this proof, we have dropped the symbol ``$\delta$" in the adaptive time-step ``$h^\delta_n$" to ease the notation. Solving the recurrence relation, we get 
\begin{align} \label{bound moments - finite - eq1}
	|\hat X^K_{\t}|^2 &\leq  |\hat X^K_0|^2 + 2  \alpha \left( \sum_{n=0}^{n_t -1} |\hat X^K_{t_n}|^2 h_n + |\bar X^K_{t_n-\tau}|^2 h_n  \right)+  2 \beta \t \nonumber \\
	&+2 \sum_{n=0}^{n_t -1} \<\hat X^K_{t_n}+f(\hat X^K_{t_n},\bar X^K_{t_n-\tau})h_n,g(\hat X^K_{t_n},\bar X^K_{t_n-\tau})\Delta W_n \>+ \sum_{n=0}^{n_t -1}|g(\hat X^K_{t_n},\bar X^K_{t_n-\tau})\Delta W_n|^2. 
\end{align}
Similarly, the continuous approximate solution verifies
\begin{align}\label{bound moments - finite - eq2}
	|X^K_t|^2 &\leq |\hat X^K_{\underline t}|^2 + 2 (t-\underline{t}) \alpha(|\hat X^K_{\underline t}|^2 + |\bar X^K_{\underline t -\tau}|^2 )+  2 (t-\underline{t}) \beta \nonumber \\
	&+2  \<\hat X^K_{\underline t}+f(\hat X^K_{\underline t},\bar X^K_{\underline t -\tau})(t-\underline{t}),g(\hat X^K_{\underline t},\bar X^K_{\underline t -\tau})(W_t - W_{\underline{t}}) \>+|g(\hat X^K_{\underline t},\bar X^K_{\underline t -\tau})(W_t - W_{\underline{t}})|^2. 
\end{align}
Substituting \eqref{bound moments - finite - eq1} into \eqref{bound moments - finite - eq2} yields
\begin{align*}
	|X^K_t|^2 &\leq  |\hat X^K_0|^2 + 2  \alpha \left( \sum_{n=0}^{n_t -1} |\hat X^K_{t_n}|^2 h_n + |\bar X^K_{t_n-\tau}|^2 h_n + |\hat X^K_{\underline t}|^2 (t-\underline{t}) + |\bar X^K_{\underline t -\tau}|^2 (t-\underline{t})  \right)+  2 \beta t \\
	&+2 \sum_{n=0}^{n_t -1} \<\hat X^K_{t_n}+f(\hat X^K_{t_n},\bar X^K_{t_n-\tau})h_n,g(\hat X^K_{t_n},\bar X^K_{t_n-\tau})\Delta W_n \>+ \sum_{n=0}^{n_t -1}|g(\hat X^K_{t_n},\bar X^K_{t_n-\tau})\Delta W_n|^2 \\
	&+ 2  \<\hat X^K_{\underline t}+f(\hat X^K_{\underline t},\bar X^K_{\underline t -\tau})(t-\underline{t}),g(\hat X^K_{\underline t},\bar X^K_{\underline t -\tau})(W_t - W_{\underline{t}}) \>+|g(\hat X^K_{\underline t},\bar X^K_{\underline t -\tau})(W_t - W_{\underline{t}})|^2.
\end{align*}
Using the step processes $\bar X^K$ and $\tilde X^K$ defined previously, the second summand on the RHS of the equation above, can be expressed as a Riemann integral. Similarly the forth and the sixth terms can be written as an It\^{o} integral, i.e. 
\begin{align*}
	|X^K_t|^2 &\leq  | X^K_0| + 2 \alpha  \int_0^t (|\bar X^K_s|^2 + |\tilde X^K_{s-\tau}|^2) ds  +  2 \beta t \\
			&+  2\int_0^t \<\bar X^K_s+f(\bar X^K_s,\tilde X^K_{s - \tau})[h(\bar X^K_u)I_{[0,\t)}(u)+(t-\t)I_{[\t,t]}(u)] ,g(\bar X^K_s,\tilde X^K_{s - \tau}) dW_s \>  \\
			&+ \sum_{n=0}^{n_t -1}|g(\bar X^K_{t_n},\tilde X^K_{t_n-\tau})\Delta W_n|^2 + |g(\bar X^K_t,\tilde X^K_{t-\tau})(W_t - W_{\underline{t}})|^2.
\end{align*}
Hence, we have
\begin{align*}
	|X^K_t|^p &\leq 6^{p/2-1} \Bigg\{ |X^K_0|^p + \left(2\alpha \int_0^t (|\bar X^K_s|^2 + |\tilde X^K_{s-\tau}|^2) ds \right)^{p/2} +  (2 \beta t)^{p/2} \\
	&+ \left|2 \int_0^t \<\bar X^K_s+f(\bar X^K_s,\tilde X^K_{s - \tau})[h(\bar X^K_u)I_{[0,\t)}(u)+(t-\t)I_{[\t,t]}(u)] ,g(\bar X^K_s,\tilde X^K_{s - \tau}) dW_s \> \right|^{p/2} \\
	&+ \left(\sum_{n=0}^{n_t -1}|g(\bar X^K_{t_n},\tilde X^K_{t_n-\tau})\Delta W_n|^2 \right)^{p/2} +|g(\bar X^K_t,\tilde X^K_{t-\tau})(W_t - W_{\underline{t}})|^p \Bigg\} .
\end{align*}
Taking the expectation of the supremum, one has 
$$	\E \left[ \sup_{0 \leq s \leq t} | X_s^K|^p \right] \leq 6^{p/2-1} (I_1 + I_2 + I_3 + I_4),$$
where
\begin{align*}
I_1 &:=\E| X^K_0|^p + \E \left[ \left(2  \alpha \int_0^t (|\bar X^K_s|^2 + |\tilde X^K_{s-\tau}|^2) ds \right)^{p/2} \right] +  (2 \beta t)^{p/2}; \\
I_2 &:=\E \left[\sup_{0 \leq s \leq t}   \left|2 \int_0^s \<\bar X^K_u+f(\bar X^K_u,\tilde X^K_{u - \tau})[h(\bar X^K_u)I_{[0,\s)}(u)+(s-\s)I_{[\s,s]}(u)] ,g(\bar X^K_u,\tilde X^K_{u - \tau}) dW_u \> \right|^{p/2} \right]; \\
I_3 &:= \E \left[ \left(\sum_{n=0}^{n_t -1}|g(\bar X^K_{t_n},\tilde X^K_{t_n-\tau})\Delta W_n|^2 \right)^{p/2} \right]; \\
I_4 &:= \E \left[ \sup_{0 \leq s \leq t} |g(\bar X^K_s,\tilde X^K_{s-\tau})(W_s - W_{\underline{s}})|^p \right].
\end{align*}
Now we will establish bounds for each of the four terms above. In the remainder of the proof, $C$ is positive constants, independent of $K,$ that may change from line to line.

Using H\"{o}lder's inequality, we have
\begin{align*}
I_1 &\leq \E| X^K_0|^p + (2  \alpha)^{p/2}  T^{p/2-1}2^{p/2-1}  \int_0^t \E [|\bar X^K_s|^p + |\tilde X^K_{s-\tau}|^p]  ds  +  (2 \beta T)^{p/2} \\
&\leq C \int_0^t \E \left[\sup_{0 \leq u \leq s} |X^K_u|^p \right]ds  + C.  
\end{align*}
By the Burkholder-Davis-Gundy (BDG) inequality we obtain
$$I_2 \leq 2^{p/2}C \E \left[\left( \int_0^t |(\bar X^K_u+f(\bar X^K_u,\tilde X^K_{u - \tau})[h(\bar X^K_u)I_{[0,\t)}(u)+(t-\t)I_{[\t,t]}(u)]) g(\bar X^K_u,\tilde X^K_{u - \tau})|^2 du \right)^{p/4} \right]$$
An application of the H\"{o}lder inequality yields that
\begin{equation} \label{I2 - finite time - eq1}
	I_2 \leq 2^{\ff p 2}T^{\ff p 4-1}C \E \left[ \int_0^t \left|\bar X^K_u+f(\bar X^K_u,\tilde X^K_{u - \tau})[h(\bar X^K_u)I_{[0,\t)}(u)+(t-\t)I_{[\t,t]}(u)]\right|^{\ff p 2} ||g(\bar X^K_u,\tilde X^K_{u - \tau})||^{\ff p 2} du  \right]
\end{equation}
Now, we bound the integrand of the integral above. Using condition \eqref{h1} we obtain
\begin{align*}
	|\bar X^K_u &+f(\bar X^K_u,\tilde X^K_{u - \tau})[h(\bar X^K_u)I_{[0,\t)}(u)+(t-\t)I_{[\t,t]}(u)]|^2=  \\
	&= |\bar X^K_u|^2 + 2 [h(\bar X^K_u)I_{[0,\t)}(u)+(t-\t)I_{[\t,t]}(u)] \Big[ \<\bar X^K_u, f(\bar X^K_u,\tilde X^K_{u - \tau})\>    \\
	&+ \ff 1 2 [h(\bar X^K_u)I_{[0,\t)}(u)+(t-\t)I_{[\t,t]}(u)] |f(\bar X^K_u,\tilde X^K_{u - \tau})|^2 \Big]  \\
	&\leq |\bar X^K_u|^2 + 2 [h(\bar X^K_u)I_{[0,\t)}(u)+(t-\t)I_{[\t,t]}(u)] \left[ \alpha \left(|\bar X^K_u|^2 + |\tilde X^K_{u - \tau}|^2 \right) + \beta  \right] \\
	&= (1+ 2 \alpha T)|\bar X^K_u|^2 + 2 \alpha T |\tilde X^K_{u-\tau}|^2 + 2\beta T.
\end{align*}
This implies
\begin{align*}
|\bar X^K_u &+f(\bar X^K_u,\tilde X^K_{u - \tau})[h(\bar X^K_u)I_{[0,\t)}(u)+(t-\t)I_{[\t,t]}(u)]|^{p/2}  \\
&\leq 3^{p/4-1}\left[ (1+ 2 \alpha T)^{p/4}|\bar X^K_u|^{p/2} + (2 \alpha T)^{p/4} |\tilde X^K_{u-\tau}|^{p/2} + (2\beta T)^{p/4}  \right]  \\
	&\leq C\left( |\bar X^K_u|^{p/2} +  |\tilde X^K_{u-\tau}|^{p/2} + 1  \right).
\end{align*}
Also by condition \eqref{g LG}  one can see that
\begin{align*}
	||g(\bar X^K_u,\tilde X^K_{u - \tau})||^{p/2} &= \left( ||g(\bar X^K_u,\tilde X^K_{u - \tau})||^2 \right)^{p/4} \leq \left[ \alpha \left(|\bar X^K_u|^2 + |\tilde X^K_{u - \tau}|^2 \right) + \beta \right]^{p/4} \\
	&\leq C \left( |\bar X^K_u|^{p/2} + |\tilde X^K_{u - \tau}|^{p/2} + 1 \right).
\end{align*}
Substituting the last two inequalities into \eqref{I2 - finite time - eq1}, we obtain
\begin{align*}
	I_2 &  \leq  C \E \left[  \int_0^t \left(1+ |\bar X^K_u|^p + |\tilde X^K_{u - \tau}|^p \right) du \right] \\
	&\leq  C + C  \left(  \int_0^t \E \left[ \sup_{0 \leq u \leq s} |X^K_u|^p \right] ds \right).
\end{align*}
Now we will bound $I_3.$ Note that $t_n$ is a stopping time of the filtration $\{\F^W_t\}.$ Define 
$$\F_{t_n}:=\{A \in \F : A \cap \{t_n \leq t \} \in \F^W_t\}.$$ 
By the strong Markov property of the Brownian motion, $\{B_u := W_{t_n+u} - W_{t_n}, u \geq 0\}$ is a standard Brownian motion independent of $\F_{t_n}$ (page 86, Theorem 6.16 in \cite{Karatzas}). Thus
$$\E[\sup_{0 \leq u \leq s}|W_{t_n+u} - W_{t_n}|^p | \F_{t_n}] = \E[\sup_{0 \leq u \leq s}|B_u|^p]\leq C s^{p/2}.$$ 
This implies
\begin{equation} \label{I3}
	\E[\sup_{t_n \leq u \leq t_{n+1}}|W_u -W_{t_n}|^p | \F_{t_n}]  \leq C h_n^{p/2} .  
\end{equation}	  

Combining Jensen's inequality and equation \eqref{I3}, we arrive at
\begin{align*}
	I_3 &\leq \E \left[ \left(\sum_{n=0}^{n_t -1} || g(\bar X^K_{t_n},\tilde X^K_{t_n-\tau})||^2 |\Delta W_n|^2 \right)^{p/2} \right] = \E \left[ \left(\sum_{n=0}^{n_t -1} h_n || g(\bar X^K_{t_n},\tilde X^K_{t_n-\tau})||^2 \ff{|\Delta W_n|^2}{h_n} \right)^{p/2} \right] \\
	&\leq T^{p/2-1} \E \left[ \sum_{n=0}^{n_t -1} h_n || g(\bar X^K_{t_n},\tilde X^K_{t_n-\tau})||^p \ff{E[|\Delta W_n|^p | \F_{t_n}]}{h_n^{p/2}} \right] \leq C T^{p/2-1} \E \left[ \sum_{n=0}^{n_t -1} h_n || g(\bar X^K_{t_n},\tilde X^K_{t_n-\tau})||^p  \right] \\
	&\leq C T^{p/2-1} \E \left[ \int_0^{\underline{t}}|| g(\bar X^K_s,\tilde X^K_{s- \tau})||^p ds \right] \leq C T^{p/2-1} \E \left[ \int_0^t || g(\bar X^K_s,\tilde X^K_{s- \tau})||^p ds \right].
\end{align*}
Using condition \eqref{g LG} and H\"{o}lder's inequality, we have
\begin{align*}
	I_3 &\leq C T^{p/2-1} \E \left[ \int_0^t \left( || g(\bar X^K_s,\tilde X^K_{s- \tau})||^2 \right)^{p/2} ds \right] \leq C T^{p/2-1} \E \left[ \int_0^t \left(  \alpha(|\bar X^K_s|^2 + |\tilde X^K_{s- \tau}|^2) + \beta \right)^{p/2} ds \right] \\
	&\leq T^{p/2-1}2^{p-2}C\E \left[ \int_0^t \left(  \alpha^{p/2}(|\bar X^K_s|^p + |\bar X^K_{s- \tau}|^p) + \beta^{p/2} \right) ds \right] \\
	&\leq  C + C    \int_0^t \E \left[ \sup_{0 \leq u \leq s} |X^K_u|^p \right] ds.  
\end{align*}
For $I_4,$ using the linear condition \eqref{g LG}, we obtain
\begin{align*}
	I_4 &\leq \E \left[ \sup_{0 \leq s \leq t} |g(\bar X^K_s,\tilde X^K_{s - \tau})(W_s - W_{\underline{s}})|^p \right] \leq  \E \left[ \sup_{0 \leq s \leq t} \left\{ [ (\alpha(|\bar X^K_s|^p + |\tilde X^K_{s - \tau})|^p) + \beta] \ |(W_s - W_{\underline{s}})|^p \right\} \right] \\
	&\leq  \E \Bigg[  \sum_{n=0}^{n_t -1} [\alpha(|\bar X^K_{t_n}|^p + |\tilde X^K_{t_n-\tau}|^p) + \beta]  \E \left[ \sup_{t_n \leq s \leq t_{n+1}}  |(W_s - W_{t_n})|^{p/2} | \F_{t_n} \right]  \\
	&+ [\alpha(|\bar X^K_t|^p + |\tilde X^K_{t-\tau}|^p) + \beta] \E \left[ \sup_{\underline{t} \leq s \leq t} |(W_s - W_{\underline t})|^{p/2} | \F_{\underline{t}} \right] \Bigg]  \\
	&\leq  C + C \int_0^t E \left[\sup_{0 \leq u \leq s} | X^K_u|^p \right] ds.
\end{align*}
Adding all the bounds for $I_1$ to $I_4,$ we have that for all $t \in [0,T]$
$$	\E \left[ \sup_{0 \leq s \leq t} | X_s^K|^p \right] \leq  C + C \int_0^t E \left[\sup_{0 \leq u \leq s} | X^K_u|^p \right]  $$ 
and by the Gronwall inequality we obtain
$$\E \left[ \sup_{0 \leq t \leq T} | X_t^K|^p \right] \leq C.$$	
\end{proof}

\begin{rem}
	Note that assuming that $T$ was attainable, we  have proved the boundedness of the $p$th moments without using the auxiliary scheme. The only reason why we needed to work with a bounded scheme was to show that $\inf_{|x| \leq K} h^\delta(x)$ is strictly positive and therefore $T$ is attainable. 
\end{rem}

\begin{proof} [Proof of Theorem \ref{thm - boundedness pth moments}]
By Lemma \ref{lem - aux boundedness of the pth moments} and the Markov inequality
$$\P(\sup_{\time} |X_t| < K) = 1- \P(\sup_{\time} |X^K_t| \geq K) \geq 1 - \frac{\E[\sup_{\time} |X^K_t|^4}{K^4} = 1 - \frac{C}{K^4}. $$
Thus 
$$ \lim_{K \rightarrow \infty} \P(\sup_{\time} |X_t| < K) = 1,$$
This means that $\sup_{\time} |X_t| < \infty$ a.s., i.e. for almost all $\omega \in \Omega$ there exist a $K_{\omega}$ such that 
\begin{equation} \label{final proof I}
	\sup_{\time} |X_t(\omega)| \leq K_{\omega}.
\end{equation}
Since $h^\delta$ is continuous and strictly positive, $\inf_{|x| \leq K_{\omega}} h^\delta(x) >0.$ This implies that for almost every $\omega \in \Omega$
$$ \liminf_{n \rightarrow \infty} h_n^\delta(\omega) = \liminf_{n \rightarrow \infty} h^\delta(X_{t_n}(\omega))  \neq 0,$$
so $\lim_{n \rightarrow \infty} t_n(\omega) = \sum_{n=0}^{\infty} h_n^\delta(\omega)= \infty$ a.s. and $T$ is attainable. Also,
for all $\omega$ and all $0 < K_1 \leq K_2,$ we have
\begin{equation}\label{final proof II}
	\sup_{\time} |X^{K_1}_t(\omega)|= \min(\sup_{\time} |X_t(\omega)|,K_1) \leq \min(\sup_{\time} |X_t(\omega)|,K_2) = \sup_{\time} |X^{K_2}_t(\omega).|
\end{equation}
Equations \eqref{final proof I} and \eqref{final proof II} imply that
\begin{equation}\label{final proof III}
	\lim_{K \rightarrow \infty} \sup_{\time}|X^K_t|=\sup_{\time}|X_t| \ \text{ a.s. }
\end{equation}
This together with Lemma \ref{lem - aux boundedness of the pth moments}, yields
$$\E \left[ \sup_{0 \leq t \leq T} | X_t|^p \right]  = \lim_{K \rightarrow \infty} \E \left[ \sup_{0 \leq t \leq T} | X_t^K|^p \right] \leq C.$$
The proof is complete for $p \geq 4$. For $0 \leq p < 4$,  The required assertion follows from the H\"{o}lder inequality.\end{proof}

\subsubsection{Strong convergence of the numerical solutions}
In order to prove the strong convergence of the approximate solution \eqref{cont-time approx sol} to the exact solution of the SDDE \eqref{SDE}, we need the following lemma and corollary.
\begin{lem}\label{distance-aux-process}
	Let the SDDE \eqref{SDE} and the function $h^\delta$ satisfy Assumption \ref{assu f and g} and \ref{assu h} respectively. Assume also that the function $f$ satisfies the (global) linear growth condition, i.e. there exist a constant $C_1 \geq 0$ such that for all $x,y \in \R^m,$  
	\begin{equation} \label{f LG}
	|f(x,y)|^2 \leq C_1(|x|^2+|y|^2+1).
	\end{equation}
	Then there exists  a positive constant $C$ such that for all $t \in [0,T].$
	\begin{align}
	\E|X_t-\bar X_t|^2 &\leq C \delta T, \label{aux1} \\ 
	\E|X_t-\tilde X_t|^2 &\leq C \delta T.	\label{aux2}
	\end{align}		
\end{lem}	
\begin{proof}
	Let $t \in [0,T]$. Let $r$ be such that $t_r \leq t \leq t_{r+1}.$ Then by definition we have $X_{t_r} = \bar X_{t_r} =\bar X_t.$ Thus
	$$X_t = \bar X_t + \int_{t_r}^t f(\bar X_s, \tilde X_s)ds + \int_{t_r}^t g(\bar X_s, \tilde X_s)dW_s.$$
	This together with \eqref{f LG},\eqref{g LG}, Assumption \ref{assu h} and Theorem \ref{thm - boundedness pth moments} imply that
	\begin{align*}
	\E|X_t-\bar X_t|^2 &\leq 2\E\left|\int_{t_r}^t f(\bar X_s, \tilde X_s)ds\right|^2 + 2\E\left|\int_{t_r}^t g(\bar X_s, \tilde X_s)dW_s\right|^2 \\
	&\leq 2\E [C_1 (h_r^\delta)^2 (1+2\sup_{t_r \leq s \leq t}|X_s|^2+||\xi||)]+2\E[\alpha h_r^\delta(2\sup_{t_r \leq s \leq t}|X_s|^2 +||\xi||)+\beta] \\
	&\leq 4 (\delta T)^2 (1+\E[\sup_{t_r \leq s \leq t}|X_s|^2]+\E||\xi||)+4\alpha \delta T (\E[\sup_{t_r \leq s \leq t}|X_s|^2] +\E||\xi||)+\beta] \\
	&\leq C \delta T.
	\end{align*}
	To prove assertion \eqref{aux2}, we first prove that there is a constant $C$ such that for all $t \in [0,T]$
	\begin{equation}\label{aux3}
	\E|\tilde X_t-\bar X_t|^2 \leq C \delta T.
	\end{equation}
	Let $t \in [0,T].$ Let $k$ and $n$ be such that $t_k \leq t < t_{k+1}$ and $t_n -\tau \leq t \leq t_{n+1}-\tau$ respectively. Let $r,0\leq r\leq k$ be such that $t_{k-r} \leq t_n-\tau \leq t_{k-r+1}$. From \eqref{dis-scheme} and the definitions of the step processes $\bar X$ and $\tilde X,$ one can see that
	\begin{align*}
	\hat X_{t_k}&=\hat X_{t_{k-r}}+\sum_{i=0}^{r-1}[f(\bar X_{t_{k-r+i}},\bar X_{t_{k-r+i}-\tau})h_{k-r+i} + g(\bar X_{t_{k-r+i}},\bar X_{t_{k-r+i}-\tau})\Delta W_{k-r+i}] \\
	&=\hat X_{t_{k-r}}+\sum_{i=0}^{r-1} \int^{t_{k-r+i+1}}_{t_{k-r+i}}f(\bar X_s,\tilde X_{s-\tau})ds +\sum_{i=0}^{r-1} \int^{t_{k-r+i+1}}_{t_{k-r+i}} g(\bar X_s,\tilde X_{s-\tau})dW_s \\
	&=\hat X_{t_{k-r}}+ \int^{t_k}_{t_{k-r}}f(\bar X_s,\tilde X_{s-\tau})ds + \int^{t_k}_{t_{k-r}} g(\bar X_s,\tilde X_{s-\tau})dW_s.
	\end{align*}
	Note that $\bar X_t = \hat X_{t_k}$ and $\hat X_{t_{k-r}}=\bar X_{t_{k-r}} = \bar X_{t_n - \tau}=\tilde X_{t_n- \tau}=\tilde X_t,$ we have that
	$$\bar X_t = \tilde X_t + \int^{t_k}_{t_{k-r}}f(\bar X_s,\tilde X_{s-\tau})ds + \int^{t_k}_{t_{k-r}} g(\bar X_s,\tilde X_{s-\tau})dW_s.$$
	Also, we have that
	$$t_k - t_{k-r} \leq (t_{n+1}-\tau)-(t_n-\tau)+h^\delta_{k-r}=h^\delta_n+h^\delta_{k-r} \leq 2 \delta T.$$ 
	Therefore, by \eqref{f LG},\eqref{g LG}, Assumption \ref{assu h} and Theorem \ref{thm - boundedness pth moments} we have that
	\begin{align*}
	\E|\bar X_t &- \tilde X_t|^2 \leq 2\E \left|\int_{t_{k-r}}^{t_k} f(\bar X_s,\tilde X_{s-\tau})ds\right|^2 + 2\E \left|\int_{t_{k-r}}^{t_k} g(\bar X_s,\tilde X_{s-\tau})dW_s\right|^2  \\
	&\leq 2\E [C_1 (t_k - t_{k-r})^2 (1+2\sup_{t_k \leq s \leq t}|X_s|^2+||\xi||)]+2\E[\alpha (t_k - t_{k-r})(2\sup_{t_k \leq s \leq t}|X_s|^2 +||\xi||)+\beta] \\
	&\leq 4 (\delta T)^2 (1+\E[\sup_{t_k \leq s \leq t}|X_s|^2]+\E||\xi||)+4\alpha \delta T (\E[\sup_{t_k \leq s \leq t}|X_s|^2] +\E||\xi||)+\beta] \\
	&\leq C \delta T.
	\end{align*}
	This together with \eqref{aux1} imply that
	$$\E|X_t - \tilde X_t|^2 = \E|X_t - \bar X_t|^2 + \E|\bar X_t -\tilde X_t|^2 \leq C \delta T.$$
	\end{proof} 
In our attempt to prove the strong convergence using the local Lipschitz condition instead of the global one, we introduce the stopping times 
$$\tau_m:= \inf\{t \geq 0:|Y_t| \geq m \}, \quad \sigma_m:= \inf\{t \geq 0:|X_t| \geq m \}$$ 
and $\upsilon_m:= \tau_m \wedge \sigma_m$. As usual we set $\inf \emptyset = \infty$. In the next corollary, we relax the global linear condition imposed to $f$ in the previous lemma and use instead the local Lipschitz condition.
\begin{cor}\label{cor-distance-aux-process}
	Let the SDDE \eqref{SDE} and the function $h^\delta$ satisfy Assumption \ref{assu f and g} and \ref{assu h} respectively. Then there exists a positive constant $C_m$ such that for all $t \in [0,T].$
	\begin{align}
	\E|X_{t \wedge \upsilon_m}-\bar X_{t \wedge \upsilon_m}|^2 &\leq C_m \delta T, \\ 
	\E|X_{t \wedge \upsilon_m - \tau}-\tilde X_{t \wedge \upsilon_m -\tau}|^2 &\leq C_m \delta T.	
	\end{align}		
\end{cor}
\begin{proof}
	The processes $X_{t \wedge \upsilon_m}, \bar X_{t \wedge \upsilon_m}$ and $\tilde X_{t \wedge \upsilon_m}$ are bounded by $m$. Thus, the local Lipschitz condition \eqref{LL} implies condition \eqref{f LG}. Therefore the corollary follows directly from Lemma \ref{distance-aux-process}.
\end{proof}

\begin{thm}
	If the SDDE \eqref{SDE} and the function $h^\delta$ satisfy Assumption \ref{assu f and g} and \ref{assu h} respectively, then for all $p>0$
$$	\lim_{\delta \rightarrow 0}  \E \left[ \sup_{\time} |X_t - Y_t|^p \right]=0.$$
\end{thm}

\begin{proof}
	One can see that
	\begin{align}\label{eq1-convergence}
	\E [\sup_{0 \leq t \leq T}|Y_t - X_t|^2] &=\E [\sup_{0 \leq t \leq T}|Y_t - X_t|^2 I_{\{\tau_m>T \text{ and } \sigma_m>T \}}]+ \E [\sup_{0 \leq t \leq T}|Y_t - X_t|^2 I_{\{\tau_m \leq T \text{ or } \sigma_m \leq T \}}], \nonumber \\
	&=: R_1 + R_2,
	\end{align}
	where $I_A$ es the indicator function of the set $A.$ In order to bound $R_1$, we combine the definitions of the continuous-time approximation \eqref{cont-time approx sol} and the exact solution \eqref{SDE} to obtain 
	\begin{align*}
	|Y_{t \wedge \upsilon_m} &- X_{t \wedge \upsilon_m}|^2  \\
	&= \left| \int_0^{t \wedge \upsilon_m}[f(Y_s,Y_{s- \tau})-f(\hat X_s, \tilde X_{s-\tau})]ds + \int_0^{t \wedge \upsilon_m}[g(Y_s,Y_{s- \tau})-g(\hat X_s, \tilde X_{s-\tau})]dW_s \right|^2 \\
	&\leq 2T \int_0^{t \wedge \upsilon_m}|f(Y_s,Y_{s- \tau})-f(\hat X_s, \tilde X_{s-\tau})|^2ds + 2 \left|\int_0^{t \wedge \upsilon_m}[g(Y_s,Y_{s- \tau})-g(\hat X_s, \tilde X_{s-\tau})]dW_s \right|^2
	\end{align*}
	Thus, for any $t_1 \leq T,$ 
	\begin{align*}
	\E &[\sup_{0 \leq t \leq t_1}|Y_{t \wedge \upsilon_m} - X_{t \wedge \upsilon_m}|^2] \\
	&\leq 2T\E\left[\int_0^{t \wedge \upsilon_m}|f(Y_s,Y_{s - \tau})-f(\hat X_s, \tilde X_{s-\tau})|^2ds\right]+8\E\left[\int_0^{t \wedge \upsilon_m}|g(Y_s,Y_{s - \tau})-g(\hat X_s, \tilde X_{s-\tau})|^2ds\right],
	\end{align*}
	where we have used the Doob martingale inequality in the second summand. Using the local Lipschitz condition \eqref{LL} in the RHS of the previous equation and then, adding and subtracting $X_t$ twice yields
	\begin{align*}
	\E &[\sup_{0 \leq t \leq t_1}|Y_{t \wedge \upsilon_m} - X_{t \wedge \upsilon_m}|^2] \\
	&\leq  C_m \left( \int_0^{t_1}\E|Y_{s \wedge \upsilon_m}-X_{s \wedge \upsilon_m}|^2 ds + \int_0^{t_1}\E|Y_{s \wedge \upsilon_m -\tau}-X_{s \wedge \upsilon_m -\tau}|^2 ds \right) \\
	&+  C_m \left( \int_0^{t_1}\E|X_{s \wedge \upsilon_m}-\bar X_{s \wedge \upsilon_m}|^2 ds + \int_0^{t_1}\E|X_{s \wedge \upsilon_m -\tau}-\tilde X_{s \wedge \upsilon_m -\tau}|^2 ds \right), 
	\end{align*}
	where $C_m$ is a positive constant that depends on $T$ and $m$.
	By Corollary \ref{cor-distance-aux-process}, we obtain
	\begin{align*}
	\E &[\sup_{0 \leq t \leq t_1}|Y_{t \wedge \upsilon_m} - X_{t \wedge \upsilon_m}|^2] \\
	&\leq  C_m \left( \int_0^{t_1}\E|Y_{s \wedge \upsilon_m}-X_{s \wedge \upsilon_m}|^2 ds + \int_0^{t_1}\E|Y_{s \wedge \upsilon_m -\tau}-X_{s \wedge \upsilon_m -\tau}|^2 ds \right)+C_m \delta.  \\ 
	\end{align*}
	The Gronwall inequality yields 
	$$R_1 = \E [\sup_{0 \leq t \leq T}|Y_{t \wedge \upsilon_m} - X_{t \wedge \upsilon_m}|^2] \leq C_m \delta.$$ 	
	Proceeding in exactly the same way as in \cite{Higham}, one can see that for all $\alpha,\beta,\eta, \mu > 0$ we have
	$$R_2 \leq \ff{2^{p+1} \eta  C}{p} + \ff{2(p-2) C}{p \eta^{2/(p-2)}m^p}$$
	where $\bar C$ is a positive constant. Substituting the estimates of $R_1$ and $R_2$ into \eqref{eq1-convergence}, we obtain
	$$\E [\sup_{0 \leq t \leq T}|Y_t - X_t|^2] \leq C_m \delta + \ff{2^{p+1} \eta  C}{p} + \ff{2(p-2) C}{p \eta^{2/(p-2)}m^p}.$$
	Now, given any $\epsilon >0$, we can find an $\eta$ sufficiently small so
	$$\ff{2^{p+1} \eta  C}{p} < \ff \epsilon 3,$$
	and then $m$ large enough so
	$$\ff{2(p-2) C}{p \eta^{2/(p-2)}m^p} < \ff \epsilon 3,$$
	and finally $\delta$ small enough such that 
	$$\delta C_m < \ff \epsilon 3.$$
	The proof is complete. 		
\end{proof}
\color{black}

\subsection{Order of convergence}
Now we investigate the order of convergence of the adaptive EM numerical solutions.

\begin{assu} \label{assu to prove order of convergence}
There exists a constant $L>0$ such that for all $x,y,\bar x, \bar y \in \R^m,$ $f$ satisfies the one-sided Lipschitz condition
\begin{equation}\label{f one-sided GL}
2\<x - \bar x, f(x,y)-f(\bar x, \bar y) \> \leq  L(|x-\bar x|^2 +|y - \bar y|^2)
\end{equation} 
and $g$ satisfies the (global) Lipschitz condition 
\begin{equation} \label{g GL}
||g(x,y) - g(\bar x, \bar y)||^2 \leq L (|x- \bar x|^2 +|y - \bar y |^2).
\end{equation} 
In addition $f$ satisfies the polynomial growth Lipschitz condition: there exist constants $\gamma, \lambda, q > 0$ such that for all $x,y,\bar x, \bar y \in \R^m$
\begin{equation} \label{f Pol growth cond}
|f(x,y) - f(\bar x, \bar y)| \leq (\gamma(|x|^q + |y|^q + |\bar x|^q + |\bar y|^q) + \lambda) (|x- \bar x|+|y - \bar y|).
\end{equation}	
Furthermore, for any $s,t \in [-\tau,0]$ and $q>0,$ there exists a positive constant $\Lambda$ such that 
\begin{equation} \label{initial condition}
	\E||\xi(t) - \xi(s)|| \leq \Lambda|t-s|^q.
\end{equation}
\end{assu}

\begin{thm} \label{theo - order of convergence}
If the SDDE \eqref{SDE} satisfies Assumption \ref{assu to prove order of convergence} and the time-step function $h$ satisfies  Assumption \ref{assu h}, then for all $p > 0,$ there exists a positive constant $C$ independent of $\delta$ such that
$$\E \left[ \sup_{\time}|X_t - Y_t|^p \right] \leq C \delta ^{p/2} .   $$
\end{thm}

\begin{proof}
	The proof is similar to that of SDEs given in \cite{Giles}. We only give the proof for $p \geq 4$; the result for $0 \leq p <4$ follows from H\"{o}lder's inequality.
	Define $e_t:=Y_t - X_t, 0 \leq t  \leq T.$ Hence
	$$e_t = \int_0^t [f(Y_s, Y_{s-\tau})-f(\bar X_s, \tilde X_{s-\tau})]ds + \int_0^t [g(Y_s, Y_{s-\tau})-g(\bar X_s, \tilde X_{s-\tau})]dW_s.$$
	Applying It\^{o}'s formula we obtain
	\begin{align} \label{applying ito to error in order of convergence theo}
		|e_t|^2 &\leq 2 \int_0^t \<e_s,f(Y_s, Y_{s-\tau})-f(\bar X_s, \tilde X_{s-\tau})\> ds + \int_0^t |g(Y_s, Y_{s-\tau})-g(\bar X_s, \tilde X_{s-\tau})|^2ds \nonumber \\
		&+ 2 \int_0^t \<e_s, (g(Y_s, Y_{s-\tau})-g(\bar X_s, \tilde X_{s-\tau}))dW_s \> \nonumber \\
		&\leq 2 \int_0^t \<e_s,f(Y_s, Y_{s-\tau})-f(X_s, X_{s-\tau})\> ds + 2 \int_0^t \<e_s,f(X_s, X_{s-\tau})-f(\bar X_s, \tilde X_{s-\tau})\> ds \nonumber  \\
		&+  \int_0^t |g(Y_s, Y_{s-\tau})-g(\bar X_s, \tilde X_{s-\tau})|^2ds + 2 \int_0^t \<e_s, (g(Y_s, Y_{s-\tau})-g(\bar X_s, \tilde X_{s-\tau}))dW_s \>.
	\end{align}
Using condition \eqref{f one-sided GL} we get
\begin{equation} \label{c1 in order of convergence theo}
	2\<e_s, f(Y_s, Y_{s-\tau})-f(X_s, X_{s-\tau})\> \leq L(|Y_s-X_s|^2+ |Y_{s-\tau}- X_{s-\tau}|^2)= L(|e_s|^2+ |e_{s-\tau}|^2).
\end{equation} 
Condition \eqref{f Pol growth cond} implies that
\begin{align} \label{c2 in order of convergence theo}
	|\<e_s,f(X_s, X_{s-\tau})&-f(\bar X_s, \tilde X_{s-\tau})\>| \leq |e_s| \ |f(X_s, X_{s-\tau})-f(\bar X_s, \tilde X_{s-\tau})| \nonumber\\
	&\leq |e_s| Q(X_s, X_{s-\tau}, \bar X_s, \tilde X_{s-\tau})(|X_s - \bar X_s| + |X_{s - \tau} - \tilde X_{s-\tau}| ) \nonumber \\
	&\leq \ff 1 2 |e_s|^2 + \ff 1 2 Q(X_s, X_{s-\tau}, \bar X_s, \tilde X_{s-\tau})^2 \ 2(|X_s - \bar X_s|^2 + |X_{s - \tau} - \tilde X_{s - \tau}|^2 ), 
\end{align} 
where $Q(x, y, \bar x, \bar y):= \gamma(|x|^q + |y|^q + |\bar x|^q + |\bar y|^q) + \lambda.$
In addition, condition \eqref{g GL} implies that
\begin{align} \label{c3 in order of convergence theo}
	||g(Y_s, Y_{s-\tau})-g(\bar X_s, \tilde X_{s-\tau})||^2 &\leq L (|Y_s - \bar X_s|^2 + |Y_{s-\tau} - \tilde X_{s-\tau}|^2) \nonumber \\
	&= L (|Y_s - X_s+X_s - \bar X_s|^2 + |Y_{s-\tau} - X_{s-\tau} + X_{s-\tau} - \tilde X_{s-\tau}|^2) \nonumber \\
	&\leq  2 L(|e_s|^2 + |e_{s-\tau}|^2 + |X_s - \bar X_s|^2 +|X_{s-\tau} - \tilde X_{s-\tau}|^2 ).
\end{align}
Substituting \eqref{c1 in order of convergence theo}, \eqref{c2 in order of convergence theo} and \eqref{c3 in order of convergence theo} in \eqref{applying ito to error in order of convergence theo}, we have
	\begin{align*}
		|e_t|^2 &\leq  \int_0^t \big[(3L+1) |e_s|^2 + 3L |e_{s - \tau}|^2\big]ds \\
		 &+ 2 \int_0^t [Q(X_s, X_{s-\tau}, \bar X_s, \tilde X_{s-\tau})^2 + L] (|X_s - \bar X_s|^2 + |X_{s - \tau} - \tilde X_{s-\tau}|^2 ) ds \\
		 &+ 2 \int_0^t \<e_s, (g(Y_s, Y_{s-\tau})-g(\bar X_s, \tilde X_{s-\tau}))dW_s \>.
	\end{align*}
Using H\"{o}lder's inequality yields
\begin{align*}
	|e_t|^p &\leq (6T)^{p/2 -1} \int_0^t((3L+1)^{p/2}|e_s|^p +(2L)^{p/2}|e_{s-\tau}|^p)ds \\
	&+ (3T)^{p/2-1} 2^{p/2} \int_0^t [Q(X_s, X_{s-\tau}, \bar X_s, \tilde X_{s-\tau})+L]^{p/2}(|X_s-\bar X_s|^p +|X_{s-\tau} - \tilde X_{s-\tau}|^p)ds \\
	&+ 3^{p/2-1} 2^{p/2} \left| \int_0^t \<e_s, (g(Y_s, Y_{s-\tau})-g(\bar X_s, \tilde X_{s-\tau}))dW_s \> \right|^{p/2}.
\end{align*}
In the remainder of the proof, $C$ is positive constant, independent of $\delta,$ that may change from line to line. \\
Taking the supremum on each side of the previous inequality and then the expectation yields
$$\E \left[\sup_{0 \leq s \leq t} |e_s|^p \right] \leq J_1 + J_2 + J_3,$$
where
\begin{align*}
	J_1 &:= C \int_0^t \E \left[\sup_{0 \leq u \leq s} |e_u|^p \right] ds; \\
	J_2 &:= C \int_0^t \E \left[ [Q(X_s, X_{s-\tau}, \bar X_s, \tilde X_{s-\tau})+L]^{p/2}(|X_s-\bar X_s|^p +|X_{s-\tau} - \tilde X_{s-\tau}|^p) \right] ds; \\
	J_3 &:= C \E \left[\sup_{0 \leq s \leq t} \left| \int_0^s \<e_s, (g(Y_u, Y_{u-\tau})-g(\bar X_u, \tilde X_{u-\tau}))dW_u \> \right|^{p/2} \right].
\end{align*}
For $J_2,$ by H\"{o}lder's inequality one has
\begin{equation} \label{I2 order}
	J_2 \leq C \int_0^t \left( \E \left[ [Q(X_s, X_{s-\tau}, \bar X_s, \tilde X_{s-\tau})+L]^{p} \right] \  \E \left[(|X_s-\bar X_s|^{2p} +|X_{s-\tau} - \tilde X_{s-\tau}|^{2p}) \right] \right)^{1/2} ds.
\end{equation} 
By Theorem \ref{thm - boundedness pth moments} there exists a constant $C$ such that
\begin{equation} \label {casa}
	\E \left[ [Q(X_s, X_{s-\tau}, \bar X_s, \tilde X_{s-\tau})+L]^{p} \right] \leq C.
\end{equation}
Let $\underline{s}:=\max\{t_n: t_n \leq s\}.$ From \eqref{cont-time approx sol}, we can write
$$X_s - \bar X_s = f(\bar X_{\underline s}, \tilde X_{\underline{s}-\tau})(s - \underline s) + g(\bar X_{\underline s}, \tilde X_{\underline{s}-\tau})(W_s - W_{\underline s}).$$
Thus, by H\"{o}lder inequality
\begin{align} \label{oso1}
	\E |X_s &- \bar X_s |^{2p} = \E|f(\bar X_{\underline s}, \tilde X_{\underline{s}-\tau})(s - \underline s) + g(\bar X_{\underline s}, \tilde X_{\underline{s}-\tau})(W_s - W_{\underline s})|^{2p} \nonumber\\
	&\leq 2^{2p-1} \E|f(\bar X_{\underline s}, \tilde X_{\underline{s}-\tau})(s - \underline s)|^{2p} + 2^{2p-1} \E|g(\bar X_{\underline s}, \tilde X_{\underline{s}-\tau})(W_s - W_{\underline s})|^{2p} \nonumber \\
	&\leq 2^{2p-1} (\E[f(\bar X_{\underline s}, \tilde X_{\underline{s}-\tau})]^{4p} \E[(s - \underline s)^{4p})^{1/2} + 2^{2p-1} (\E[g(\bar X_{\underline s}, \tilde X_{\underline{s}-\tau})]^{4p} \E[(W_s - W_{\underline s})^{4p}])^{1/2}.
\end{align}
By Assumption \ref{assu h} we have
\begin{equation}\label{oso2}
E[(s - \underline s)^{4p}] \leq \E[(h_{\underline{s}}^\delta)^{4p}] \leq (\delta T)^{4p} \leq \delta^{2p}T^{4p}
\end{equation}
and by condition \eqref{I3}, we get
\begin{equation}\label{oso3}
	\E[(W_s - W_{\underline s})^{4p}] \leq C(\delta T)^{2p}.
\end{equation}
Also it follows from the global Lipschitz condition \ref{g GL} that 
\begin{align}\label{oso4}
	||g(\bar X_{\underline s}, \tilde X_{\underline{s}-\tau})||^{4p} &\leq \ff 1 {2^{2p}} K^{2p} (|\bar X_{\underline s}|^2+ |\tilde X_{\underline{ s}-\tau}|^2)^{2p} +C \\
	&\leq C(|\bar X_{\underline s}|^{4p} + |\tilde X_{\underline{s}-\tau}|^{4p} +1) \nonumber
\end{align}
and from the polynomial growth condition that
\begin{align}\label{oso5}
	|f(\bar X_{\underline s}, \tilde X_{\underline{s}-\tau})|^{4p} &\leq \left[ (\gamma(|\bar X_{\underline s}|^q + |\tilde X_{\underline{s} -\tau})|^q) + \mu) (|\bar X_{\underline s}|+|\tilde X_{\underline{s}-\tau})|) + f(0,0) \right]^{4p} \\
	&\leq C( |\bar X_{\underline s}|^{4p(q+1)} + |\tilde X_{\underline{s}-\tau})|^{4p(q+1)} +1), \nonumber
\end{align}
so by Theorem \ref{thm - boundedness pth moments}, there exists a constant $C$  such that 
$$\E[|f(\bar X_{\underline s}, \tilde X_{\underline{s}-\tau})|^{4p}] \leq C  \mbox{ and } \E[|g(\bar X_{\underline s}, \tilde X_{\underline{s}-\tau})|^{4p}]  \leq  C.$$ 

Substituting these last two expressions together with \eqref{oso2} and \eqref{oso3} into \eqref{oso1}, we obtain 
\begin{equation}\label{oso6}
\E |X_s - \bar X_s |^{2p} \leq C \delta^p.
\end{equation}
Using \eqref{oso4} and \eqref{oso5}, and proceeding in exactly the same way as in Lemma \ref{distance-aux-process}, yields $\E |X_{s-\tau} - \tilde X_{s-\tau} |^{2p} \leq C \delta^p.$ Using this fact together with \eqref{oso6} and \eqref{casa} in  \eqref{I2 order}, we obtain that $J_2 \leq C \delta^{p/2}.$ 

Now we estimate $J_3.$ By the BDG and H\"{o}lder's inequalities one can see that
\begin{align*}
	J_3 &\leq C \E \left[ \left( \int_0^t |e_s|^2\  |(g(Y_s, Y_{s-\tau})-g(\bar X_s, \tilde X_{s-\tau}))|^2 ds  \right)^{p/4} \right] \\
	&\leq C \E\left[\int_0^t |e_s|^{p/2} (|\bar X_s - Y_s|^{p/2} + |\tilde X_{s-\tau} - Y_{s-\tau}|^{p/2})ds \right]  \\
	&\leq C \E\left[\int_0^t \ff 1 2 |e_s|^p + |\bar X_s - Y_s|^p + |\tilde X_{s-\tau} - Y_{s-\tau}|^p ds \right]   \\
	&\leq C \E\left[\int_0^t  |e_s|^p + (|\bar X_s - X_s|^p + |X_s - Y_s|^p + |\tilde X_{s-\tau} - X_{s-\tau}|^p + |X_{s - \tau} - Y_{s-\tau}|^p)ds \right] \\
	&\leq C \E\left[\int_0^t  |e_s|^p +  |e_{s-\tau}|^p + (|\bar X_s - X_s|^p +  |\tilde X_{s-\tau} - X_{s-\tau}|^p )ds \right]. 
\end{align*}
By the same argument we used with $J_2$ we know that
$$\E \left[(|\bar X_s - X_s|^p +  |\tilde X_{s-\tau} - X_{s-\tau}|^p ) \right] \leq C \delta^{p/2}.$$ 
Thus
$$J_3 \leq C \int_0^t \E \left[\sup_{0 \leq u \leq s} |e_u|^p  \right]ds +C \delta^{p/2}.$$ 
Collecting the bounds for $J_1, J_2$ and $J_3,$ we conclude that there exist  a constant $C$  such that
$$\E\ \left[\sup_{\time} |e_t|^p \right] \leq  C \int_0^t \E \left[\sup_{0 \leq u \leq s} |e_u|^p  \right]ds +  C \delta^{p/2}.$$ 
The required assertion follows from the Gronwall inequality.
\end{proof}

\section{Convergence of the numerical solutions on infinite time interval} \label{infinite}
In this section we will study the convergence of the numerical solutions on the time interval $[0, \infty).$ The assumptions will be stronger than the ones on the finite time interval. 

\begin{assu} \label{assu f and g - infinite time} The functions $f$ and $g$ satisfy the local Lipschitz condition: for every $R>0$ there exists a positive constant $C_R$ such that
	\begin{equation} \label{LL - infinite time}
		| f(x,y)-f(\bar x, \bar y)| +||g(x,y)-g(\bar x,\bar y)|| \leq C_R (|x- \bar x|+|y- \bar y|)
	\end{equation}
	for all $x,y,\bar x, \bar y \in \R^m$ with $|x|,|y|,|\bar x|, |\bar y| \leq R.$ Furthermore, there exists constants $\alpha_1 > \alpha_2 \geq 0$ and $\beta > 0,$ such that for all $x,y \in \R^m,$ $f$ satisfies the dissipative one-sided linear growth condition: 
	\begin{equation} \label{f one-sided LG - infinite time}
		\<x,f(x,y)\> \leq -\alpha_1|x|^2 + \alpha_2|y|^2+ \beta, 
	\end{equation}
	and $g$ is globally bounded: 
	\begin{equation} \label{g LG infinite time}
		||g(x,y)||^2 \leq \beta.
	\end{equation}	
\end{assu}

\begin{assu} \label{assu h - infinite time}
	For every $\delta,$ the time step function $h^\delta:\R^m \rightarrow \R^+,$ is continuous and uniformly bounded by $h^\delta_{max}$, where  $h^\delta_{max} \in (0,\infty).$ \\ 
	Furthermore, there exist constants $\alpha_1 > \alpha_2 \geq 0$ and $\beta > 0,$ such that for all $x,y \in \R^m.$
	\begin{equation}\label{h1 - infinite time}
		\<x,f(x,y)\> + \ff 1 2 h^\delta(x)|f(x,y)|^2 \leq -\alpha_1|x|^2 + \alpha_2|y|^2+ \beta.
	\end{equation}
\end{assu}

\subsection{The boundedness of the $p$th moments of the exact and the numerical solutions}

\subsubsection{Exact solution}

\begin{lem}
	If the SDDE \eqref{SDE} satisfies Assumption \ref{assu f and g - infinite time}, then there exists a positive constant $C$ such that for all $t \geq 0$
	\begin{equation}
		\E \left[ |Y_t|^p \right] \leq C.
	\end{equation} 
\end{lem}

\begin{proof}
The  proof is standard, we omit it here. 
\end{proof}

\subsubsection{Adaptive EM numerical solutions} 

The proof about attainability given for the finite time interval, is valid for the infinite time interval $[-\tau, \infty).$    

\begin{thm} \label{thm - boundedness pth moments - infinite time} 
	If the SDE \eqref{SDE} and the function $h^\delta$ satisfy Assumption \ref{assu f and g - infinite time} and \ref{assu h - infinite time} respectively, then for all $p>0$ there exists a constant $C$ dependent on $h_{max}, \beta, \alpha_1, \alpha_2$ and $p,$ but independent of $\delta$ and $t,$ such that for all $t \geq 0,$
	\begin{equation}
		\E \left[  |X_t|^p \right] \leq C.
	\end{equation} 
\end{thm}

\begin{proof}
	The proof is given for $p \geq 4.$ For $0 < p <4,$ the result holds from H\"{o}lder's inequality. Fix $t$ and define $\underline{t}:=\max\{t_n: t_n \leq t\},$ \ $\underline{\hat t}:=\max\{t_n: t_n \leq t - \tau \}$  and $n_t:=\max\{n:t_n \leq t\}.$ Taking squared norms in \eqref{dis-scheme}, we have that for $n=0$ to $n = n_t,$ 
	\begin{align*} 
		|\hat X_{t_{n+1}}|^2 &= |\hat X_{t_n}|^2 + 2 h_n (\<\hat X_{t_n},f(\hat X_{t_n},\bar X_{t_n - \tau})\>+ \ff 1 2 h_n |f(\hat X_{t_n},\bar X_{t_n - \tau})|^2) \\
		&+2 \< \hat X_{t_n}+f(\hat X_{t_n},\bar X_{t_n - \tau})h_n,g(\hat X_{t_n},\bar X_{t_n - \tau})\Delta W_n \>+|g(\hat X_{t_n},\bar X_{t_n - \tau})\Delta W_n|^2. 	
	\end{align*}
Note that, since it is irrelevant in this proof, we have dropped the term ``$\delta$" in the adaptive time-step ``$h^\delta_n$" to ease the notation.Using conditions \eqref{h1 - infinite time} and \eqref{g LG infinite time}, we obtain
\begin{align*} 
	|\hat X_{t_{n+1}}|^2 &\leq |\hat X_{t_n}|^2 - 2 h_n \alpha_1|\hat X_{t_n}|^2 + 2 h_n \alpha_2 |\bar X_{t_n - \tau}|^2 +  2 h_n \beta  \\
	&+2  \<\hat X_{t_n}+f(\hat X_{t_n},\bar X_{t_n - \tau})h_n,g(\hat X_{t_n},\bar X_{t_n - \tau})\Delta W_n \>+ \beta|\Delta W_n|^2.	
\end{align*}
Multiplying both sides by $e^{2 \alpha_1 t_{n+1}}$ yields
\begin{align*} 
	e^{2 \alpha_1 t_{n+1}}&|\hat X_{t_{n+1}}|^2 \leq  e^{2 \alpha_1 t_{n+1}}|\hat X_{t_n}|^2 - 2 h_n \alpha_1 e^{2 \alpha_1 t_{n+1}}|\hat X_{t_n}|^2 + 2 h_n \alpha_2 e^{2 \alpha_1 t_{n+1}}|\bar X_{t_n - \tau}|^2   \\
	&+  2 h_n \beta e^{2 \alpha_1 t_{n+1}} +2 e^{2 \alpha_1 t_{n+1}}  \<\hat X_{t_n}+f(\hat X_{t_n},\bar X_{t_n - \tau})h_n,g(\hat X_{t_n},\bar X_{t_n - \tau})\Delta W_n \> + e^{2 \alpha_1 t_{n+1}}\beta|\Delta W_n|^2.	
\end{align*}
Now, taking into account that $t_{n+1} = t_n + h_n$ and using the fact that for all $x \in \R,$ $1 + x \leq e^x$ with $x = -2 h_n \alpha_1,$ we obtain
\begin{align*} 
	e^{2 \alpha_1 t_{n+1}}&|\hat X_{t_{n+1}}|^2 \leq  e^{2 \alpha_1 t_n}|\hat X_{t_n}|^2 + 2 h_n \alpha_2 e^{2 \alpha_1 t_{n+1}}|\bar X_{t_n - \tau}|^2  +  2 h_n \beta e^{2 \alpha_1 t_{n+1}} \\
	&+2 e^{2 \alpha_1 t_{n+1}}  \<\hat X_{t_n}+f(\hat X_{t_n},\bar X_{t_n - \tau})h_n,g(\hat X_{t_n},\bar X_{t_n - \tau})\Delta W_n \> + e^{2 \alpha_1 t_{n+1}}\beta|\Delta W_n|^2.	
\end{align*}
Solving the recurrence, we have
\begin{align} 
	e&^{2\alpha_1 \t}|\hat X_{\t}|^2 \leq |\hat X_0|^2 +   2 \alpha_2 \sum_{n=0}^{n_t-1} e^{2 \alpha_1 t_{n+1}}|\bar X_{t_n - \tau}|^2 h_n +  2 \beta \sum_{n=0}^{n_t-1}  e^{2 \alpha_1 t_{n+1}}h_n \nonumber \\
	&+2 \sum_{n=0}^{n_t-1} e^{2 \alpha_1t_{n+1}}  \< \hat X_{t_n}+f(\hat X_{t_n},\bar X_{t_n - \tau})h_n,g(\hat X_{t_n},\bar X_{t_n - \tau})\Delta W_n \> + \beta \sum_{n=0}^{n_t-1} e^{2 \alpha_1  t_{n+1}}|\Delta W_n|^2.	
\end{align}
Similarly for the partial time step from $\t$ to $t,$ we get
\begin{align} 
	e^{2 \alpha_1 t}&| X_t|^2 \leq  e^{2 \alpha_1 \t}|\hat X_{\t}|^2 + 2 (t-\t) \alpha_2 e^{2 \alpha_1 t}|\bar X_{\t -\tau}|^2  +  2 (t-\t) \beta e^{2 \alpha_1 t} \nonumber \\
	&+2 e^{2 \alpha_1 t}  \< \hat X_{\t}+f(\hat X_{\t},\bar X_{\t -\tau})h_n,g(\hat X_{\t},\bar X_{\t -\tau})(W_t - W_{\t}) \> + e^{2 \alpha_1 t}\beta|(W_t - W_{\t})|^2.	
\end{align}
Substituting the penultimate inequality into the last one, we obtain
\begin{align*} 
	e^{2\alpha_1 t}| X_t|^2 &\leq  |X_0|^2 +   2 \alpha_2 \sum_{n=0}^{n_t-1} e^{2 \alpha_1 t_{n+1}}|\bar X_{t_n - \tau}|^2 |h_n + 2 \alpha_2 e^{2 \alpha_1 t}|\bar X_{t_n - \tau}|^2  (t - \t) \\
	&+2 \beta \sum_{n=0}^{n_t-1}  e^{2 \alpha_1 t_{n+1}}h_n  +2\beta e^{2 \alpha_1 t}(t-\t) \\
	&+2 \sum_{n=0}^{n_t-1} e^{2 \alpha_1 t_{n+1}}  \< \hat X_{t_n}+f(\hat X_{t_n},\bar X_{t_n - \tau})h_n,g(\hat X_{t_n},\bar X_{t_n - \tau})\Delta W_n \> \\
	+&\beta \sum_{n=0}^{n_t-1} e^{2 \alpha_1 t_{n+1}}|\Delta W_n|^2 + e^{2 \alpha_1 t}\beta|(W_t - W_{\t})|^2 \nonumber \\
	&+2 e^{2 \alpha_1 t}  \< \hat X_{\t}+f(\hat X_{\t},\bar X_{\t -\tau})(t-\t),g(\hat X_{\t},\bar X_{\t -\tau})(W_t - W_{\t}) \> .	
\end{align*}
Since $t_{n+1} \leq t_n + h_{max}$ and $t \leq \t + h_{max},$ we can take the common factor $e^{2 \alpha_1 h_{max}}$ out in the equation above. The processes $\bar X$ and $\tilde X$, defined in \eqref{dis-scheme} and \eqref{tilde-process} respectively, are a simple processes, so we express the second and the third terms in the RHS of the previous equation as a Riemann integral. The same for the fourth and fifth terms. Similarly, the sixth and ninth terms can be written together as a (pathwise) It\^{o} integral,  
\begin{align*}
e^{2 \alpha_1 t}&|X_t|^2 \leq |X_0|^2 +  e^{2 \alpha_1 h_{max}} \Bigg\{  \int_0^t e^{2 \alpha_1 s}|\tilde X_{s-\tau}|^2 ds +   2 \beta \int_0^t e^{2 \alpha_1 s} ds \\
	&+ 2\int_0^t e^{2 \alpha_1 s} \<\bar X_s+f(\bar X_s,\tilde X_{s - \tau})[h(\bar X_s)I_{[0,\t)}(s)+(t-\t)I_{[\t,t]}(s)] ,g(\bar X_s,\tilde X_{s - \tau}) dW_s \> \\
	&+ \beta \sum_{n=0}^{n_t-1} e^{2\alpha_1 t_{n}}|\Delta W_n|^2 + e^{2 \alpha_1 \t}\beta|(W_t - W_{\t})|^2 \Bigg\}.
\end{align*}
Now, raising to the power $p/2$, using H\"{o}lder's inequality and taking the expectation of the supremum, we obtain
\begin{equation}
e^{p\alpha_1 t} \E \left[\sup_{0 \leq s \leq t} | X_t|^p \right] \leq 6^{p/2-1}e^{p \alpha_1 h_{max}} (H_1 + H_2 + H_3 + H_4),
\end{equation}
where
\begin{align*}
	H_1 &:=\E| X_0|^p + \E \left[ \left( 2 \alpha_2 \int_0^t e^{2 \alpha_1s} |\tilde X_{s-\tau}|^2ds \right)^{p/2} \right] +  \left( 2 \beta \int_0^t e^{2 \alpha_1 s} ds \right)^{p/2}; \\
	H_2 &:=\E \Bigg[\sup_{0 \leq s \leq t} \Bigg|2 \int_0^s e^{2 \alpha_1 u} \<\bar X_u+f(\bar X_u,\tilde X_{u - \tau})[h(\bar X_u)I_{[0,\s)}(u) \\
	 & \quad \quad \quad \quad \quad \quad \quad \quad \quad \quad \quad \quad \quad \quad \quad \quad \quad \quad \quad +(s-\s)I_{[\s,s]}(u)] ,g(\bar X_s,\tilde X_{u - \tau}) dW_u \> \Bigg|^{p/2} \Bigg]; \\
	H_3 &:= \E \left[ \left(\beta \sum_{n=0}^{n_t-1} e^{2\alpha_1 t_{n}}|\Delta W_n|^2 \right)^{p/2} \right]; \\
	H_4 &:= \beta^{p/2} e^{p\alpha_1 t} \E[\sup_{0 \leq s \leq t} |(W_s - W_{\s})|^p].
\end{align*}
Now we will establish bounds for each of the four terms above. In the remainder of the proof, $C$ is a  positive constant that may depend on $\beta, \alpha_1, \alpha_2, h_{max}$ and $p$, but independent of $t,$ that may change from line to line.
We start by bounding $H_1$. 
\begin{align*}
	H_1 &\leq \E| X_0|^p +  \E \left[ \left( 2 \alpha_2 \sup_{-\tau \leq s \leq t}|X_s|^2 \int_0^t e^{2 \alpha_1s} ds \right)^{p/2} \right] +  \left( 2 \beta \int_0^t e^{2\alpha_1 s} ds \right)^{p/2} \\
	&\leq \E| X_0|^p +   \left(\ff{ \alpha_2}{ \alpha_1}\right)^{p/2} \E \left[\sup_{-\tau \leq s \leq t}|X_s|^p \right] e^{ \alpha_1 p t}  +  \left(\ff {2 \beta}{2 \alpha_1}\right)^{p/2}  e^{\alpha_1 p t}  \\
	&\leq e^{\alpha_1 p t}\left(C+\left(\ff{ \alpha_2}{ \alpha_1}\right)^{p/2} \E \left[\sup_{0 \leq s \leq t}|X_s|^p \right] \right).
\end{align*}

For $H_2,$ the BDG inequality and condition \eqref{g LG infinite time} yields
$$H_2 \leq 2^{p/2}\beta^{p/4}C \E \left[\left( \int_0^t e^{4(\alpha_1 - \alpha_2)s} |(\bar X_s+f(\bar X_s,\tilde X_{s - \tau})[h(\bar X_s)I_{[0,\t)}(s)+(t-\t)I_{[\t,t]}(s)])|^2 ds \right)^{p/4} \right].$$
Since $e^{4(\alpha_1-\alpha_2)s} = e^{2(\alpha_1-\alpha_2)\ff {p-4} {p} s} e^{2(\alpha_1-\alpha_2)(1 + \ff 4 p) s},$ by H\"{o}lder's inequality, we get 
\begin{align*}
	\bigg( &\int_0^t e^{4(\alpha_1 - \alpha_2)s} |(\bar X_s+f(\bar X_s,\tilde X_{s - \tau})[h(\bar X_s)I_{[0,\t)}(s)+(t-\t)I_{[\t,t]}(s)])|^2 ds \bigg)^{p/4} \\
	  &\leq \left(\int_0^t e^{2(\alpha_1-\alpha_2) s} ds \right)^{\ff {p-4} {4}} \\
	  &\times \int_0^t e^{(\alpha_1-\alpha_2) \ff {p+4} 2 s}|(\bar X_s+f(\bar X_s,\tilde X_{s - \tau})[h(\bar X_s)I_{[0,\t)}(s)+(t-\t)I_{[\t,t]}(s)])|^{p/2} ds. 
\end{align*} 
Using Assumption \eqref{assu h - infinite time}, we obtain
\begin{align*}
	|\bar X_s &+f(\bar X_s,\tilde X_{s - \tau})[h(\bar X_s)I_{[0,\t)}(s)+(t-\t)I_{[\t,t]}(s)]|^2  \\
	&\leq |\bar X_s|^2 + 2 [h(\bar X_s)I_{[0,\t)}(s)+(t-\t)I_{[\t,t]}(s)] \left( -\alpha_1 |\bar X_s|^2 + \alpha_2 |\tilde X_{s - \tau}|^2  + \beta  \right) \\
	&\leq |\bar X_s|^2 + 2 h_{max} \left( \alpha_2 |\tilde X_{s-\tau}|^2  + \beta  \right).
\end{align*}
Therefore,
\begin{align*}
	H_2 &\leq \E \Bigg[C \left(\int_0^t e^{2 \alpha_1  s} ds \right)^{\ff {p-4} {4}} \\
	&\times \int_0^t e^{\alpha_1 \ff {p+4} 2 s}\left\{ |\bar X_s|^{p/2} + (2 h_{max} \alpha_2)^{p/4}|\tilde X_{s-\tau}|^{p/2}+(2\beta h_{max})^{p/4} \right\}   ds \Bigg].
\end{align*}
We can write the previous inequality as $H_2 \leq H_{21} + H_{22} +H_{23},$
where
\begin{align*}
	H_{21}&:= C \E[\sup_{0 \leq s \leq t}|X_s|^{p/2}] \left(\int_0^t e^{2\alpha_1 s} ds \right)^{\ff {p-4} {4}} \int_0^t e^{ \alpha_1 \ff {p+4} 2 s} ds; \\
	H_{22}&:= C(2 h_{max} \alpha_2)^{p/4}\E[\sup_{-\tau \leq s \leq t}| X_s|^{p/2}] \left(\int_0^t e^{2\alpha_1 s} ds \right)^{\ff {p-4} {4}} \left(\int_0^t e^{\alpha_1 \ff {p+4} 2 s}  ds\right); \\
	H_{23}&:= C(2 h_{max} \alpha_2)^{p/4} \left(\int_0^t e^{2\alpha_1 s} ds \right)^{\ff {p-4} {4}} \left(\int_0^t e^{\alpha_1 \ff {p+4} 2 s}  ds\right).
\end{align*}
Since,
\begin{align*}
	\left(\int_0^t e^{2 \alpha_1 s} ds \right)^{\ff {p-4} {4}} &\int_0^t e^{\alpha_1 \ff {p+4} 2 s}= \ff {e^{\alpha_1 (p-4)t}-1}{(2 \alpha_1 )^{\ff {p-4} 4}} \  \cdot \  \ff{e^{ \alpha_1 \ff {p+4} 2} t -1}{\alpha_1 \ff{p+4} 2} \\
	&\leq \ff{e^{\alpha_1 p t}}{\alpha_1 \ff{p+4} 2 (2 \alpha_1 )^{\ff {p-4} 4}}\leq  C e^{ \alpha_1 p t},
\end{align*}
we arrive at
\begin{align*}
H_2 &\leq C \E[\sup_{0 \leq s \leq t}|X_s|^{p/2}] e^{\alpha_1 p t} +  C \E[\sup_{-\tau \leq s \leq t}|X_s|^{p/2}] e^{\alpha_1 p t} +  C e^{\alpha_1 p t} \\
&= e^{ \alpha_1 p t}(C \E[\sup_{0 \leq s \leq t}|X_s|^{p/2}]+  C). 
\end{align*}
Using the elementary inequality $ab \leq \ff 1 {2\gamma} a^2 + \ff \gamma 2 b^2$ for all $\gamma \in \R^{+}$ and all $a,b \in \R$ with $a=C$ and $b = \E[\sup_{0 \leq s \leq t}|X_s|^{p/2},$ and later Jensen's inequality, we get 
\begin{align*}
	C \E[\sup_{0 \leq s \leq t}|X_s|^{p/2} ]\leq \ff 1 {2\gamma} C^2 + \ff \gamma 2 (\E[\sup_{0 \leq s \leq t}|X_s|^{p/2}])^2 \leq \ff 1 {2\gamma} C^2 + \ff \gamma 2 \E[\sup_{0 \leq s \leq t}|X_s|^p].
\end{align*}
Therefore,
\begin{equation} \label{I2 infinite time}
	H_2 \leq e^{\alpha_1 p t}(\ff \gamma 2 \E[\sup_{0 \leq s \leq t}|X_s|^p]+ C_{\gamma}),
\end{equation}
where the ``$\gamma$" in $C_{\gamma}$ is to emphasise that this constant depends also on $\gamma$ and is not fixed yet. 

Now we will estimate $H_3.$ By the discrete H\"{o}lder's inequality we obtain
\begin{align*}
	\left|\sum_{n=0}^{n_t-1} e^{2 \alpha_1 t_n} |\Delta W_n|^2\right| &=  \left|\sum_{n=0}^{n_t-1} \left(h_n^{\ff {p-2} p} e^{2 \alpha_1 t_n \ff{p-2} p} \right) \left( h_n^{\ff 2 p} e^{\ff {4 \alpha_1 t_n}{p}}  \ff{|\Delta W_n|^2}{h_n} \right) \right| \\
	&\leq \left( \sum_{n=0}^{n_t-1} h_n e^{2 \alpha_1 t_n} \right)^{\ff {p-2} p} \left( \sum_{n=0}^{n_t-1} h_n e^{\ff {2 \alpha_1 t_n}{p}}\ff{|\Delta W_n|^p}{h_n^{p/2}} \right)^{\ff 2 p}.
\end{align*}
By \eqref{I3} we can derive that
\begin{align*}
	H_3 &\leq \E \left[ \beta^{p/2} \left(\sum_{n=0}^{n_t-1} h_n e^{2 \alpha_1 t_n} \right)^{\ff {p-2} 2} \sum_{n=0}^{n_t-1} h_n e^{2 \alpha_1 t_n} \ff {|\Delta W_n|^p}{h_n^{p/2}} \right] \\
	&\leq \beta^{p/2} \left(\int_0^{\t} e^{2 \alpha_1 s}ds \right)^{\ff {p-2} 2} C \int_0^{\t} e^{2 \alpha_1 s}ds  \leq Ce^{2 \alpha_1 t}.
\end{align*}
Using \eqref{I3} again, we have that
$$H_4 \leq \beta^{p/2}e^{\alpha_1 pt}C h_{max}^{p/2} \leq Ce^{\alpha_1 pt}.$$
Collecting together the bounds for $H_1$, $H_2$ $H_3$ and $H_4$, we obtain
$$e^{p \alpha_1 t} \E[\sup_{0 \leq s\leq t} |X_s|^p] \leq e^{p \alpha_1 t}(C_{\gamma} +  \ff \gamma 2  \E[\sup_{0 \leq s\leq t} |X_s|^p])+  \left(\ff{\alpha_2}{\alpha_1} \right)^{p/2} \E[\sup_{0 \leq s\leq t} |X_s|^p]). $$
Noting that the constant $C$ is independent of $t$, $0 \leq (\alpha_2/\alpha_1)^{p/2} < 1$ and taking $\gamma$  small enough such that $\ff \gamma 2  < 1-(\alpha_2/\alpha_1)^{p/2}$, the required assertion follows.
\end{proof}

\section{Almost sure exponential stability for SDDEs} \label{stab SDDEs}

 It was shown in \cite{Wu} that among other conditions, when the drift function satisfy the linear growth condition, the Euler-Maruyama approximate solution is a.s. exponentially stable. However, when the drift function satisfies the less restrictive one-sided linear growth condition, the EM solution needs not longer to be stable. It was proved in the same paper that the BEM solution maintains the stability. But it's well known that the BEM method is much more computationally expensive than explicit methods such as the adaptive EM method. Therefore, it is desirable to find explicit methods that provide numerical solutions that maintain the stability of the exact solution. Our goal in this section is to show that the adaptive solution can be a.s. exponentially stable for some SDDEs where the EM breaks down.

\begin{assu} \label{f and g - stability} The functions $f$ and $g$ satisfy the local Lipschitz condition: for every $R>0$ there exists a positive constant $C_R$ such that
	\begin{equation} \label{LL - stability}
		| f(x,y)-f(\bar x, \bar y)| +||g(x,y)-g(\bar x,\bar y)|| \leq C_R (|x- \bar x|+|y- \bar y|)
	\end{equation}
	for all $x,y,\bar x, \bar y \in \R^m$ with $|x|,|y|,|\bar x|, |\bar y| \leq R.$ Furthermore, there exist constants $\alpha_1, \alpha_2$ and $\beta$ satisfying
	 \begin{equation}\label{alpha1 - alpha2}
	 \alpha_1 > 2\alpha_2 \geq 0 \text { and } \beta > 0,
	 \end{equation} 
	 such that for all $x,y \in \R^m,$ $f$ satisfies  
	\begin{equation} \label{f and g - infinite time}
		\<x,f(x,y)\> + \ff 1 2 ||g(x,y)||^2 \leq -\alpha_1|x|^2 + \alpha_2|y|^2. 
	\end{equation}
\end{assu}

Under this assumption, the SDDE \eqref{SDE} has a unique solution.

\subsection{Counterexample (SDDE)}

We now return to the counterexample  \ref{SDDE counterexample}.

Let $X_k$ be defined by \eqref{EM counterexample'}
The following lemma proves a much stronger result that $X_k$ is not almost sure exponential  stable. It shows that the set in which the EM solution grows at a geometric rate has positive probability. 
\begin{lem} \label{lemma-counter}
	Consider the EM approximate solution \eqref{EM counterexample'} to the SDE \eqref{SDDE counterexample}. Then
	\begin{equation}	\label{eq0 - lemma counter'}
		\P\left(|X_k| \geq \frac{2^{k+3}}{\sqrt{\Delta}}, \ \forall k \geq 1 \right) > 0.  	
	\end{equation}	
\end{lem}
The following proof is based on the counterexample's proof given in \cite{Higham2}.

\begin{proof}
	First we show that if $|X_1| \geq 2^4/\sqrt{\Delta},$ then 
\begin{equation}\label{eq1 - lemma counter}
	\P\left(|X_k| \geq \frac{2^{k+3}}{\sqrt{\Delta}}, \ \forall k \geq 1 \right) \geq \exp\left(-4e^{-2/\sqrt{\Delta}}\right).
\end{equation}	  	
We start by proving the following fact:
\begin{equation} \label{counterexample implication}	
|X_k| \geq \frac{2^{k+3}}{\sqrt{\Delta}} \quad \text{and} \quad |\Delta W_k| \leq 2^k \quad \text{imply} \quad |X_{k+1}| \geq \frac{2^{k+4}}{\sqrt{\Delta}}.
\end{equation}
To prove \eqref{counterexample implication}, assume that $|X_k| \geq \frac{2^{k+3}}{\sqrt{\Delta}}.$ Then		\begin{align*}
		|X_{k+1}| &\geq |X_k|\left| |X_k|^2\Delta -|1 +2 \Delta +1/2 \sin(X_{k-1}) \Delta + \sqrt{2} \cos(X_{k-1}) \Delta W_k|\right| \\
		&\geq |X_k|\left| |X_k|^2\Delta -(|1| +|2 \Delta| +|1/2 \Delta| + |\sqrt{2} \Delta W_k|)\right| \\
		&\geq \ff{2^{k+3}}{\sqrt{\Delta}} (2^{2k+6} - 6- \sqrt{2} 2^k) 
		\geq \ff{2^{k+4}}{\sqrt{\Delta}} (2^{2k+5} - 3- \sqrt{2} 2^{k-1}) \\
		&\geq \frac{2^{k+4}}{\sqrt{\Delta}}.	
	\end{align*}
Now, from \eqref{counterexample implication}, given that $|X_1| \geq 2^4/\sqrt{\Delta}$, for any integer $K \geq 0,$ the event that $\{|X_k| \geq 2^{k+3}/\sqrt{\Delta}, \forall 1 \leq k \leq K\}$ contains the event that $\{|W_k| \leq 2^k, \forall 1 \leq k \leq K \}$. Since $\{\Delta W_k\}$ are independent, we have
$$\P \left(|X_k| \geq \frac{2^{k+3}}{\sqrt{\Delta}}, \ \forall 1 \leq k \leq K \right) \geq \prod_{k=1}^K \P(|\Delta W_k| \leq 2^k).  $$
In order to prove \eqref{eq1 - lemma counter}, the rest of the proof is identical to the one in Lemma 3.1 in \cite{Higham2}.
To obtain the final result, Equation \eqref{eq1 - lemma counter}, we need to prove that $\mathbb P (|X_1| \geq 2^4/\sqrt{\Delta})>0.$ But this is true since $X_1$ is a normal random variable and  for a normal random variable $X$ with density function $f$, we have that for all $a \in \mathbb R,$ $\mathbb P(X \geq a)=\int_a^{\infty} f(x)dx >0.$
\end{proof} 
In contrast to the standard EM solution, now we will see that the adaptive EM solution, maintains the stability of the exact solution of SDDE \eqref{SDDE counterexample}. But previous to that, we need to impose more assumptions.
\begin{assu} \label{h - stability}
	For every $\delta,$ the time step function $h^\delta:\R \rightarrow \R^+,$ is continuous and there exist constants $\alpha_1 > \alpha_2 \geq 0$ and $\beta > 0,$ such that for all $x,y \in \R^m,$
	\begin{equation}\label{eq h - stability}
		\<x,f(x,y)\> + \ff 1 2 h^\delta(x)|f(x,y)|^2 + \ff d 2 ||g(x,y)||^2 \leq -\alpha_1|x|^2 + \alpha_2 \ff {\min(h^\delta(y),h^\delta(x))}{h^\delta(x)} |y|^2,
	\end{equation}
where $d$ is the dimension of the Brownian motion in the SDDE \eqref{SDE}. 
	Furthermore, the function $h^\delta$ is uniformly bounded by the real numbers $0 < h^\delta_{\min} < h^\delta_{\max} < 1$, where $h^\delta_{\max}$ is small enough such that  
	\begin{equation}\label{h_max}
		2 \alpha_2 e^{2 \alpha_1 h_{\max}} < \alpha_1.
	\end{equation}
	
\end{assu}
Note that condition \eqref{eq h - stability} implies condition \eqref{f and g - infinite time} with the same values of $\alpha_1$ and $\alpha_2$. An example of function $h^\delta$ that satisfies condition \eqref{eq h - stability} for the SDDE \eqref{SDDE counterexample} is 
\begin{equation} \label{hs-counter}
	h^\delta(x):= \left( \ff 1 {25} I_{\{|x|<1\}} + 0.25 I_{\{|x|\geq 1\}}\ff{|x|^2}{\max(1,|f(x,y)|^2)} \right) \delta.
\end{equation} 

The following is the main result of this section.
\begin{thm}\label{a.s. exp stab SDDE}
	Consider the SDDE \eqref{SDE} with a $d$-dimensional Brownian motion. If $f$ and $g$ satisfy Assumption \ref{f and g - stability} and $h^\delta$ satisfies Assumption \ref{h - stability}, then the adaptive approximate solution \eq{dis-scheme} is almost sure exponentially stable, i.e. there exists a $\lambda > 0$ such that
	$$\limsup_{n \rightarrow \infty} \ff {\log|\hat X_{t_n}|}{t_n} \leq -  \lambda  \text{ a.s.}$$ 
\end{thm}
Before proving Theorem \ref{a.s. exp stab SDDE}, we show that the SDDE \eq{SDDE counterexample} satisfies Assumption \ref{f and g - stability}
$$\<x,f(x,y)\> + \ff 1 2 |g(x,y)|^2 = -2x^2-x^4+ \ff 1 2 \sin(y) x^2  +x^2 \cos^2(y) \leq - \ff 1 2 x^2 .$$
In order to show that $h^\delta$ satisfies \eqref{eq h - stability} for the SDDE \eq{SDDE counterexample}, we substitute \eqref{hs-counter} into \eqref{eq h - stability} and differentiate between the cases $|x|<1$ and $|x| \geq 1$. For $|x|<1$ we have
\begin{align*}
	\<x,&f(x,y)\> + \ff 1 2 h^\delta(x)|f(x,y)|^2 + \ff d 2 ||g(x,y)||^2 = -2x^2-x^4+ \ff 1 2 x^2 \sin(y) \\
	&+\ff 1 2 \ff 1 {25}\delta(4x^2+4x^4-2x^2 \sin(y)+x^6-x^4 \sin(y)+ \ff 1 4 x^2 \sin(y)) + \ff 1 2 2 x^2 \cos^2(y) \\
	&\leq \ff {-3x^2} {10} 
\end{align*}
and for $|x| \geq 1$ we have
\begin{align*}
\<x,&f(x,y)\> + \ff 1 2 h^\delta(x)|f(x,y)|^2 + \ff d 2 ||g(x,y)||^2 \\
&= -2x^2-x^4+ \ff 1 2 x^2 \sin(y) + \ff 1 2 \ff 1 4 \delta |x|^2 + \ff 1 2 2 x^2 \cos^2(y) \leq \ff {-3x^2} {8}.
\end{align*}
Thus the adaptive approximate solution of the SDDE \eq{SDDE counterexample} implemented with $h^\delta$ defined as  \eqref{hs-counter} is almost sure exponentially stable.

We will prove the theorem, but first we need the following lemma.
\begin{lem} \label{lem - stability}
Consider the SDDE \eqref{SDE} with a $d$-dimensional Brownian motion. Suppose $f$ and $g$ satisfy Assumption \ref{f and g - stability} and $h^{\delta}$ satisfies Assumption \ref{h - stability}. Let $l$ be a positive integer. Then there exists $\lambda \in (0, \alpha_1)$ such that 
\begin{align} \label{lem - eq0}
\sum_{n=1}^{l}  e^{\lambda t_n} |\hat X_{ t_n}|^2 h_n &\leq C   + C \sum_{n=1}^{l}e^{\lambda t_n} |g(\hat X_{t_n},\bar X_{t_n - \tau})|^2 (|\Delta W_n|^2 - h_nd) \nonumber \\
&+ C \sum_{n=1}^{l}e^{\lambda t_n} \langle \hat X_{t_n}+f(\hat X_{t_n},\bar X_{t_n -\tau})h_n,g(\hat X_{t_n},\bar X_{t_n - \tau})\Delta W_n \rangle \text{ a.s., }	
\end{align}
where $C$ is a positive constant dependent on $\omega \in \Omega$, the constants $\alpha_1,\alpha_2, h_{max}$ and $\lambda$, but independent of $l$ or $t_n$.
\end{lem}

\begin{proof}
From \eq{dis-scheme} and \eq{eq h - stability}, we have
\begin{align*} 
	|\hat X_{t_{n+1}}|^2 &= |\hat X_{t_n}|^2 +2h_n(\langle \hat X_{t_n},f(\hat X_{t_n},\bar X_{t_n - \tau})\rangle + \ff 1 2 h_n|f(\hat X_{t_n},\bar X_{t_n - \tau})|^2)   \\
	&+  2\langle \hat X_{t_n}+ f(\hat X_{t_n},\bar X_{t_n - \tau})h_n,g(\hat X_{t_n},\bar X_{t_n - \tau})\Delta W_n \rangle + |g(\hat X_{t_n},\bar X_{t_n - \tau}) \Delta W_n|^2 \\
	&\leq  |\hat X_{t_n}|^2 +2h_n(\langle \hat X_{t_n},f(\hat X_{t_n},\bar X_{t_n - \tau})\rangle + \ff 1 2 h_n|f(\hat X_{t_n},\bar X_{t_n - \tau})|^2 + \ff d 2 |g(\hat X_{t_n},\bar X_{t_n - \tau})|^2) \\
	&+	2\langle \hat X_{t_n}+ f(\hat X_{t_n},\bar X_{t_n - \tau})h_n,g(\hat X_{t_n},\bar X_{t_n - \tau})\Delta W_n \rangle	+ |g(\hat X_{t_n},\bar X_{t_n - \tau})|^2(|\Delta W_n|^2 - h_n d) \\
	&\leq  |\hat X_{t_n}|^2 -2 \alpha_1 h_n |\hat X_{t_n}|^2 + 2 \alpha_2 h^\delta(\bar X_{t_n - \tau}) |\bar X_{t_n - \tau}|^2 \\
	&+ 2\langle \hat X_{t_n}+ f(\hat X_{t_n},\bar X_{t_n - \tau})h_n,g(\hat X_{t_n},\bar X_{t_n - \tau})\Delta W_n \rangle + |g(\hat X_{t_n},\bar X_{t_n - \tau})|^2(|\Delta W_n|^2 - h_n d).
\end{align*}
Multiplying by $e^{\alpha_1 t_{n+1}}$ and using the fact that $1+x \leq e^x$ with $x = -h_n \alpha_1$, yields
\begin{align*}
	e^{\alpha_1 t_{n+1}}|\hat X_{t_{n+1}}|^2 &\leq  e^{\alpha_1 t_n}|\hat X_{t_n}|^2 + 2  \alpha_2 h^\delta(\bar X_{t_n - \tau}) e^{\alpha_1 t_{n+1}}  |\bar X_{t_n - \tau}|^2 \\
	&+ e^{\alpha_1 t_{n+1}}|g(\hat X_{t_n},\bar X_{t_n - \tau})|^2(|\Delta W_n|^2 - h_n d) \\
	&+ 2 e^{\alpha_1 t_{n+1}} \langle \hat X_{t_n}+ f(\hat X_{t_n},\bar X_{t_n - \tau})h_n,g(\hat X_{t_n},\bar X_{t_n - \tau})\Delta W_n \rangle.
\end{align*}
Solving the recurrence and using the bound $h_{\max}$, one can see that
\begin{align*}
	&e^{\alpha_1 t_n}|\hat X_{t_n}|^2 \leq  |X_0|^2 + e^{\alpha_1 h_{\max}} \Bigg\{ \sum_{k=0}^{n-1} e^{\alpha_1 t_k}|g(\hat X_{t_k},\bar X_{t_k - \tau})|^2(|\Delta W_k|^2 - h_k d) \\
	&+2 \alpha_2 \sum_{k=0}^{n-1} e^{\alpha_1 t_k}  |\bar X_{t_k - \tau}|^2 h^\delta(\bar X_{t_k - \tau}) + 2 \sum_{k=0}^{n-1} e^{\alpha_1 t_k} \langle \hat X_{t_k}+f(\hat X_{t_k},\bar X_{t_k - \tau})h_k,g(\hat X_{t_k},\bar X_{t_k - \tau})\Delta W_k \rangle \Bigg\}.
\end{align*}
Thus,
\begin{align*}
	|\hat X_{t_n}|^2 &\leq  e^{-\alpha_1 t_n}|X_0|^2 	+ e^{\alpha_1 h_{\max}} \Bigg\{e^{-\alpha_1 t_n} \sum_{k=0}^{n-1} e^{\alpha_1 t_k}|g(\hat X_{t_k},\bar X_{t_k - \tau})|^2(|\Delta W_k|^2 - h_k d) \\
	+& 2 \alpha_2 e^{-\alpha_1 t_n} \sum_{k=0}^{n-1} e^{\alpha_1 t_k}  |\bar X_{t_k - \tau}|^2 h^\delta(\bar X_{t_k - \tau})\\
	& + 2 e^{-\alpha_1 t_n} \sum_{k=0}^{n-1} e^{\alpha_1 t_k} \langle \hat X_{t_k}+f(\hat X_{t_k},\bar X_{t_k - \tau})h_k,g(\hat X_{t_k},\bar X_{t_k - \tau})\Delta W_k \rangle \Bigg\}.
\end{align*}
So, for any $\lambda \in (0,\alpha_1)$ we have
\begin{align}\label{lemma - eq0}
	\sum_{n=1}^{l} e^{\lambda t_n}|\hat X_{t_n}|^2 h_n &\leq  e^{-(\alpha_1-\lambda) t_n}|X_0|^2 h_n 	+ e^{\alpha_1 h_{\max}} \Bigg\{\sum_{n=0}^{l} e^{-(\alpha_1-\lambda) t_n} h_n \sum_{k=0}^{n-1} e^{\alpha_1 t_k}|g(\hat X_{t_k},\bar X_{t_k - \tau})|^2 \nonumber \\
	&(|\Delta W_k|^2 - h_k d) + 2 \alpha_2 \sum_{n=0}^{l} e^{-(\alpha_1-\lambda) t_n} h_n  \sum_{k=0}^{n-1} e^{\alpha_1 t_k}  |\bar X_{t_k - \tau}|^2 h^\delta(\bar X_{t_k - \tau}) \nonumber \\
	&+2 \sum_{n=0}^{l} e^{-(\alpha_1-\lambda) t_n} h_n \sum_{k=0}^{n-1} e^{\alpha_1 t_k} \langle \hat X_{t_k}+f(\hat X_{t_k},\bar X_{t_k - \tau})h_k,g(\hat X_{t_k},\bar X_{t_k - \tau})\Delta W_k \rangle \Bigg\}.
\end{align}
Moreover, we can see that
\begin{align*}
2 \alpha_2 e^{\alpha_1 h_{\max}} \sum_{n=1}^{l} e^{-(\alpha_1-\lambda) t_n} &h_n  \sum_{k=0}^{n-1} e^{\alpha_1 t_k}  |\bar X_{t_k - \tau}|^2 h^\delta(\bar X_{t_k - \tau})  \\
&= 2 \alpha_2 e^{\alpha_1 h_{\max}} \sum_{n=1}^{l}e^{\alpha_1 t_n}  |\bar X_{t_n - \tau}|^2 h^\delta(\bar X_{t_n - \tau})   \sum_{k=n}^{l}  e^{-(\alpha_1-\lambda) t_k} h_k.
\end{align*}
Now since the function $e^{-(\alpha_1-\lambda)s}$ is decreasing on $s$, we see that 
$$\sum_{k=n}^{l}  e^{-(\alpha_1-\lambda) t_k} h_k = \sum_{k=n}^{l}  e^{(\alpha_1-\lambda) h_k}e^{-(\alpha_1-\lambda) t_{k+1}} h_k \leq e^{(\alpha_1-\lambda)h_{\max}}\int_{t_n}^{t_l} e^{-(\alpha_1-\lambda)s} ds   \leq \ff {e^{\alpha_1 h_{\max}}} {\alpha_1 - \lambda} e^{-(\alpha_1 - \lambda)t_n}.$$
Thus
\begin{align}\label{lemma - eq0A} 
	2 \alpha_2 e^{\alpha_1 h_{\max}} \sum_{n=1}^{l} e^{-(\alpha_1-\lambda) t_n} h_n  &\sum_{k=0}^{n-1} e^{\alpha_1 t_k}  |\bar X_{t_k - \tau}|^2 h^\delta(\bar X_{t_k - \tau}) \nonumber\\
	&\leq \ff{2 \alpha_2 e^{2\alpha_1 h_{max}}}{\alpha_1 - \lambda} \left(\sum_{n=1}^{l}e^{\lambda t_n}  |\bar X_{t_n - \tau}|^2 h^\delta(\bar X_{t_n - \tau})   \right). 
\end{align}
Let $M=M(\omega)$ be such that $t_M \leq \tau < t_{M+1}.$ Then we can write
\begin{align}\label{lemma - eq0B}
\sum_{n=1}^{l} e^{\lambda t_n}  |\bar X_{t_n - \tau}|^2 h^\delta(\bar X_{t_n - \tau}) &= \sum_{n=1}^{M}e^{\lambda t_{n}}  |\bar X_{t_n-\tau}|^2 h^\delta(\bar X_{t_n - \tau}) +\sum_{n=M+1}^{l}e^{\lambda t_{n}}  |\bar X_{t_n-\tau}|^2 h^\delta(\bar X_{t_n - \tau}) \nonumber \\
&\leq C + e^{\lambda h_{\max} M}\sum_{n=1}^{l}e^{\lambda t_n}  |\hat X_{t_n}|^2 h_n,
\end{align}
Substituting Equation \eqref{lemma - eq0B} into \eqref{lemma - eq0A}, we obtain
\begin{align} \label{lemma - eq1}
2 \alpha_2 e^{\alpha_1 h_{\max}} \sum_{n=1}^{l} e^{-(\alpha_1-\lambda) t_n} h_n  \sum_{k=0}^{n-1} &e^{\alpha_1 t_k}  |\bar X_{t_k - \tau}|^2 h_k  \nonumber \\
&\leq  C + \ff{2 \alpha_2 e^{2\alpha_1 h_{max}}e^{\lambda h_{max} M}}{\alpha_1 - \lambda} \sum_{n=1}^{l}e^{\lambda t_n}  |\hat X_{t_n}|^2 h_n,
\end{align} 
Similarly we obtain
\begin{align} \label{lemma - eq2}
e^{\alpha_1 h_{\max}} \sum_{n=0}^{l} &e^{-(\alpha_1-\lambda) t_n} h_n \sum_{k=0}^{n-1} e^{\alpha_1 t_k}|g(\hat X_{t_k},\bar X_{t_k - \tau})|^2 (|\Delta W_k|^2 - h_k d) \nonumber \\
&\leq  \ff{2e^{\alpha_1 h_{\max}}}{\alpha_1-\lambda }\sum_{n=1}^{l}e^{\lambda t_n} |g(\hat X_{t_n},\bar X_{t_n - \tau})|^2 (|\Delta W_n|^2 - h_n d).
\end{align}
and
\begin{align} \label{lemma - eq3}
	e^{\alpha_1 h_{\max}} 2 \sum_{n=0}^{l} &e^{-(\alpha_1-\lambda) t_n} h_n \sum_{k=0}^{n-1} e^{\alpha_1 t_k} \langle \hat X_{t_k}+f(\hat X_{t_k},\bar X_{t_k - \tau})h_n,g(\hat X_{t_k},\bar X_{t_k - \tau})\Delta W_k \rangle  \nonumber \\
	&\leq  \ff{2 e^{2\alpha_1 h_{\max}}}{\alpha_1-\lambda }\sum_{n=1}^{l}e^{\lambda t_n} \langle \hat X_{t_n}+f(\hat X_{t_n},\bar X_{t_n - \tau})h_n,g(\hat X_{t_n},\bar X_{t_n - \tau})\Delta W_n \rangle.
\end{align}

We observe that by condition \eqref{h_max}, $h_{max}$ is such that  $0< 2 \alpha_2 e^{2\alpha_1 h_{max}} <\alpha_1$. Then by choosing $\lambda$ small enough so $0< \ff{2 \alpha_2 e^{2\alpha_1 h_{max}}e^{\lambda h_{max} M}}{\alpha_1 - \lambda} <1$ and by substituting Equations \eqref{lemma - eq1}, \eqref{lemma - eq2} and \eqref{lemma - eq3} into \eqref{lemma - eq0},  we obtain the final result.

\end{proof}

We are now in the position to give 

{\bf Proof of Theorem \ref{a.s. exp stab SDDE}}. 
From \eq{dis-scheme} and \eq{eq h - stability}, we have
\begin{align*} 
|\hat X_{t_{n+1}}|^2 &= |\hat X_{t_n}|^2 +2h_n(\langle \hat X_{t_n},f(\hat X_{t_n},\bar X_{t_n - \tau})\rangle + \ff 1 2 h_n|f(\hat X_{t_n},\bar X_{t_n - \tau})|^2)   \\
&+  2\langle \hat X_{t_n}+ f(\hat X_{t_n},\bar X_{t_n - \tau})h_n,g(\hat X_{t_n},\bar X_{t_n - \tau})\Delta W_n \rangle + |g(\hat X_{t_n},\bar X_{t_n - \tau}) \Delta W_n|^2 \\
&\leq  |\hat X_{t_n}|^2 +2h_n(\langle \hat X_{t_n},f(\hat X_{t_n},\bar X_{t_n - \tau})\rangle + \ff 1 2 h_n|f(\hat X_{t_n},\bar X_{t_n - \tau})|^2 + \ff d 2 |g(\hat X_{t_n},\bar X_{t_n - \tau})|^2) \\
&+	2\langle \hat X_{t_n}+ f(\hat X_{t_n},\bar X_{t_n - \tau})h_n,g(\hat X_{t_n},\bar X_{t_n - \tau})\Delta W_n \rangle	+ |g(\hat X_{t_n},\bar X_{t_n - \tau})|^2(|\Delta W_n|^2 - h_n d) \\
&\leq  |\hat X_{t_n}|^2 -2 \alpha_1 h_n |\hat X_{t_n}|^2 + 2 \alpha_2 h^\delta(\bar X_{t_n - \tau}) |\bar X_{t_n - \tau}|^2 \\
&+ 2\langle \hat X_{t_n}+ f(\hat X_{t_n},\bar X_{t_n - \tau})h_n,g(\hat X_{t_n},\bar X_{t_n - \tau})\Delta W_n \rangle + |g(\hat X_{t_n},\bar X_{t_n - \tau})|^2(|\Delta W_n|^2 - h_n d).
\end{align*}
Now we multiply by $e^{\lambda t_{n+1}},$ where $\lambda \in (0,\alpha_1)$ is the one from Lemma \ref{lem - stability}, which makes equation \eqref{lem - eq0} to hold true. Then using the fact that $1+x \leq e^x$ with $x = -2h_n \alpha_1$, yields
\begin{align*}
	e^{\lambda t_{n+1}}|\hat X_{t_{n+1}}|^2 &\leq  e^{\lambda t_n}|\hat X_{t_n}|^2 + 2  \alpha_2 e^{\lambda t_{n+1}}  |\bar X_{t_n - \tau}|^2 h_n 	+ e^{\lambda t_{n+1}}|g(\hat X_{t_n},\bar X_{t_n - \tau})|^2(|\Delta W_n|^2 - h_n d) \\
	&+ 2 e^{\lambda t_{n+1}} \langle \hat X_{t_n}+ f(\hat X_{t_n},\bar X_{t_n - \tau})h_n,g(\hat X_{t_n},\bar X_{t_n - \tau})\Delta W_n \rangle.
\end{align*}
Note that in the equation above we have used the fact that $e^{- h_n \alpha_1} \leq e^{- h_n \lambda}$.
Solving the recurrence and using the bound $h_{\max}$ we have
\begin{align*} 
	e^{\lambda t_n}&|\hat X_{t_n}|^2 \leq  |X_0|^2 	+ e^{\lambda h_{\max}} \Bigg\{ \sum_{k=0}^{n-1} e^{\lambda t_k}|g(\hat X_{t_k},\bar X_{t_k - \tau})|^2(|\Delta W_k|^2 - h_k d) \\
	&+2 \alpha_2 \sum_{k=0}^{n-1} e^{\lambda t_k}  |\bar X_{t_k - \tau}|^2 h^\delta(\bar X_{t_k - \tau}) + 2 \sum_{k=0}^{n-1} e^{\lambda t_k} \langle \hat X_{t_k}+f(\hat X_{t_k},\bar X_{t_k - \tau})h_k,g(\hat X_{t_k},\bar X_{t_k - \tau})\Delta W_k \rangle \Bigg\}.
\end{align*}
Using \eqref{lemma - eq0B}, we obtain
\begin{align} \label{theo-eq0-stability}
	e^{\lambda t_n}|\hat X_{t_n}|^2 &\leq  |X_0|^2 	+ e^{\lambda h_{\max}} \Bigg\{ \sum_{k=0}^{n-1} e^{\lambda t_k}|g(\hat X_{t_k},\bar X_{t_k - \tau})|^2(|\Delta W_k|^2 - h_k d) + C  \nonumber\\
	&+ e^{\lambda h_{max} M}\sum_{k=1}^{n-1}e^{\lambda t_k}  |\hat X_{t_k}|^2 h_k + 2 \sum_{k=0}^{n-1} e^{\lambda t_k} \langle \hat X_{t_k}+f(\hat X_{t_k},\bar X_{t_k - \tau})h_k,g(\hat X_{t_k},\bar X_{t_k - \tau})\Delta W_k \rangle \Bigg\}.
\end{align}
Substituting Equation \eqref{lem - eq0} (from Lemma \ref{lem - stability}) into \eqref{theo-eq0-stability} yields
\begin{align*}
	e^{\lambda t_n}|\hat X_{t_n}|^2 &\leq  |X_0|^2 	+ C + C \sum_{k=0}^{n-1} e^{\lambda t_k}|g(\hat X_{t_k},\bar X_{t_k - \tau})|^2(|\Delta W_k|^2 - h_k d) \\
	&+ C \sum_{k=0}^{n-1} e^{\lambda t_k} \langle \hat X_{t_k}+f(\hat X_{t_k},\bar X_{t_k - \tau})h_k,g(\hat X_{t_k},\bar X_{t_k - \tau})\Delta W_k \rangle \Bigg\}\\
	&\leq C + C	\{M_n +N_n\},
\end{align*}
where:
\begin{itemize}
	\item $M_n:= \sum_{k=0}^{n-1} e^{\lambda t_k}|g(\hat X_{t_k},\bar X_{t_k - \tau})|^2(|\Delta W_k|^2 - h_k d);$
	\item $N_n:= \sum_{k=0}^{n-1} e^{\lambda t_k} \langle \hat X_{t_k}+f(\hat X_{t_k},\bar X_{t_k - \tau})h_k,g(\hat X_{t_k},\bar X_{t_k - \tau})\Delta W_k \rangle$;
	\item $C$ is a positive constant (that changed from the second to the last line) dependent on $\omega \in \Omega$ and on the constants $\alpha_1,\alpha_2, h_{max}$ and $\lambda$, but not on $t_n$.
\end{itemize}	
Taking logarithms and dividing by $t_n$, it follows that
$$ \ff 1 {t_n} \log(e^{\lambda t_n} |X_{t_n}|^2) \leq \ff 1 {t_n} \log \left( C + C \{M_n +N_n\}\right).$$
We observe that
\begin{align*}
	\E[M_{n+1} | \F_{t_n}] &= \E[e^{\lambda t_n}|g(\hat X_{t_n},\bar X_{t_n - \tau})|^2(|\Delta W_n|^2 - h_n d) + M_n  | \F_{t_n}] \\
	&= e^{\lambda t_n}|g(\hat X_{t_n},\bar X_{t_n - \tau})|^2 (\E[|\Delta W_n|^2] - h_n d)+ M_n =M_n
\end{align*}
and
\begin{align*}
	\E[N_{n+1} | \F_{t_n}] &= \E[ 2 e^{\lambda t_n} \langle \hat X_{t_n}+ f(\hat X_{t_n},\bar X_{t_n - \tau})h_n,g(\hat X_{t_n},\bar X_{t_n - \tau})\Delta W_n \rangle + N_n  | \F_{t_n}] \\
	&= 2 e^{\lambda t_n} \langle \hat X_{t_n}+ f(\hat X_{t_n},\bar X_{t_n - \tau})h_n,g(\hat X_{t_n},\bar X_{t_n - \tau}) \E[\Delta W_n] \rangle + N_n =N_n.
\end{align*}
Hence $M+N$ is a local martingale with respect to $\{\F_{t_n}\}.$ Thus by the discrete semimartingale convergence theorem (see lemma 2 in \cite{Wu}), one can see that
$$\lim_{n \rightarrow \infty} (M_n+N_n) < \infty \ \text{ a.s.}$$
Therefore,
$$ \limsup_{n \rightarrow \infty} \ff 1 {t_n} \log(e^{\lambda t_n} |\hat X_{t_n}|^2) \leq 0 \text{ a.s.} $$
This is
$$ \limsup_{n \rightarrow \infty} \ff {\log|\hat X_{t_n}|}{t_n} \leq - \ff \lambda 2 \text{ a.s.} $$	
The proof is therefore complete.\hfill $\Box$

\begin{rem}
In the Wei and Giles \cite{Giles}, the almost sure exponential stability of the approximate adaptive EM solution has not been investigated.  Here we would like to point out that the adaptive EM solution of SDEs also reproduce the almost sure exponential stability as SDDEs. A similar result is achieved in \cite{Higham3} by using the more computationally expensive BEM method. Let $\{W_t\}_{t \geq 0}$ be a $d$-dimensional Brownian motion. Consider the $m$-dimensional SDE
\begin{equation} \label{SDE'}
	d \tilde Y_t = f( \tilde Y_t)dt + g( \tilde Y_t)dW_t
\end{equation}
for $t \geq 0$ where $f:\R^m \rightarrow \R^m$ and $g:\R^m \rightarrow \R^{m \times d}$ are Borel-measurable functions, and initial data $ \tilde Y_0 =\xi \in L^2_{\F_0}(\Omega;\R^m),$ i.e. $\xi$ is a $\F_0$-measurable $\mathbb R^m$-valued random variable with $E|\xi|^2 < \infty.$ In this case Assumption \ref{f and g - stability} can be written as 
\begin{assu} \label{assu f and g'} The functions $f$ and $g$ satisfy the local Lipschitz condition: for every $R>0$ there exists a positive constant $C_R$ such that
	\begin{equation} \label{LL'}
		| f(x)-f(y)| +||g(x)-g(y)|| \leq C_R(|x-y|
	\end{equation}
	for all $x,y \in \R^m$ with $|x|,|y| \leq R.$ Furthermore, there exists a constant $\alpha \geq 0$ such that for all $x \in \R^m,$ $f$ and $g$ satisfy 
	\begin{equation} \label{monotone condition}
		\<x,f(x)\> + \ff 1 2 |g(x)|^2 \leq -\alpha |x|^2, \quad \alpha >0.
	\end{equation}	
\end{assu}

Under the conditions \eqref{LL'} and \eqref{monotone condition}, the SDE \eqref{SDE'} has a unique solution (Theorem 2.3.6 in \cite{Mao}).  

In contrast to the EM solution, now we will see that the adaptive approximate solution of the SDE  preserves the stability of the exact solution.
We define the discrete-time adaptive approximate solution to the SDE \eq{SDE'} as
\begin{equation} \label{def h'}
	\tilde X_0 :=  \tilde Y_0, \quad h^\delta_n:=h^\delta(   \tilde X_{t_n}), \quad t_{n+1}:= t_n +h_n,
\end{equation}
and
\begin{equation} \label{dis-time approx sol II'}
	 \tilde  X_{t_{n+1}}:=   \tilde X_{t_n} + f(  \tilde X_{t_n})h^\delta_n + g( \tilde  X_{t_n})\Delta W_n,
\end{equation}
where $\Delta W_n:=W_{t_{n+1}}-W_{t_n}.$
Now, Assumption \ref{h - stability} takes the form 
\begin{assu}\label{h assu for stability}
The time-step function $h^\delta$ satisfies
\begin{equation}  \label{eq h assu for stability}
\langle x,f(x) \rangle + \ff d 2 |g(x)|^2 + \ff 1 2 h^\delta(x)|f(x)|^2 \leq -\alpha|x|^2, \quad \alpha >0	
\end{equation}	
for all $x \in \R.$	Furthermore, $h^\delta$ is uniformly bounded by the real number $h_{\max}^\delta \in (0,\infty).$
	
\end{assu}

\begin{thm}\label{a.s. exp stab SDE}
Consider the SDE \eqref{SDE'}. If $f$ and $g$ satisfy Assumption \ref{assu f and g'} and $h^\delta$ satisfies Assumption \ref{h assu for stability}, then the adaptive approximate solution \eq{dis-time approx sol II'} is almost sure exponentially stable, i.e. there exists $\lambda >0$ such that
$$\limsup_{n \rightarrow \infty} \ff {\log| \tilde X_{t_n}|}{t_n} \leq - \lambda  \text{ a.s.}$$ 
\end{thm}	
\end{rem}

\section{Simulations} \label{simulations}
In this section we present simulations which illustrate the results discussed in Section \ref{stab SDDEs}. Consider the SDDE \eqref{SDDE counterexample} with $\tau = 1$ and initial condition $Y(t)=100, -1 \leq t \leq 0.$ We simulated in Matlab paths of the EM solution of the SDDE \eqref{SDDE counterexample} using different step sizes, $\Delta$. As we have seen in section \ref{stab SDDEs} there is a positive probability that the EM solution explodes. In Table \ref{sim_EM} we present six different simulations of the EM solution for $\Delta = \num{2e-3}.$ We observe in simulations 1,3,4 and 5 the EM solution explodes. 

\begin{table}[H]
	\footnotesize
	\caption{Six simulations of the EM solution for $\Delta = \num{2e-3}$}
	\begin{tabular}{|c|r|r|r|r|r|r|r|r|r|r|r|r|}
		\hline
		\textbf{Time} & 0 & \num{2e-3} & \num{4e-3} & \num{6e-3} & \num{8e-3} & \num{10e-3}& \num{12e-3}& \num{14e-3}& \num{16e-3}& \num{18e-3}& \num{20e-3} \\ \hline
		\textbf{Sim 1} & 100 & 101.1 & 107.4 & -141.1 & 418.1 & \num{-1.4e+4} & \num{5.7e+8} & \num{-3.7e+22} & \num{1.1e+64} & \num{-2.3e+188} & Inf  \\ \hline
		\textbf{Sim 2} & 100 & -98   & 88.97 & -50.99 & -24.51& -21.33 & -19.37 & -17.29 & -16.15 & -15.13 & -14.87  \\ \hline
		\textbf{Sim 3} & 100 & -101.3 & 109.6 & -150.1 & 525.68 & \num{-2.8e+4} & \num{4.6e+9} & \num{-2e+25} & \num{1.6e+72} & \num{-8.3e+212} & Inf \\ \hline
		\textbf{Sim 4} & 100 & -101.9 & 108.5 & -143.9 & 452.6 & \num{-1.8e+4} & \num{1.2e+9} & \num{-3.3e+23} & \num{7.3e+66} & \num{-7.9e+196} & Inf \\ \hline
		\textbf{Sim 5} & 100 & -101.9 & 108.5 & -143.9 & 452.6 & \num{-1.8e+4} & \num{1.2e+9} & \num{-3.3e+23} & \num{7.3e+66} & \num{-7.9e+196} & Inf \\ \hline
		\textbf{Sim 6} & 100 & -99 & 91.8 & -63.44 & -11.65 & -11.03 & -10.87 & -10.27 & -10.17 & -9.91 & -10 \\ \hline
	\end{tabular}
	\label{sim_EM}
\end{table}

In Figure \ref{EM_explosion}, we graphed the logarithm of EM solution presented in Table \ref{sim_EM}. 
\begin{figure} [H]
	\caption{Simulations of the logarithm of the EM solution for $\Delta = \num{2e-3}$  \label{EM_explosion}}
	\includegraphics[scale=0.6]{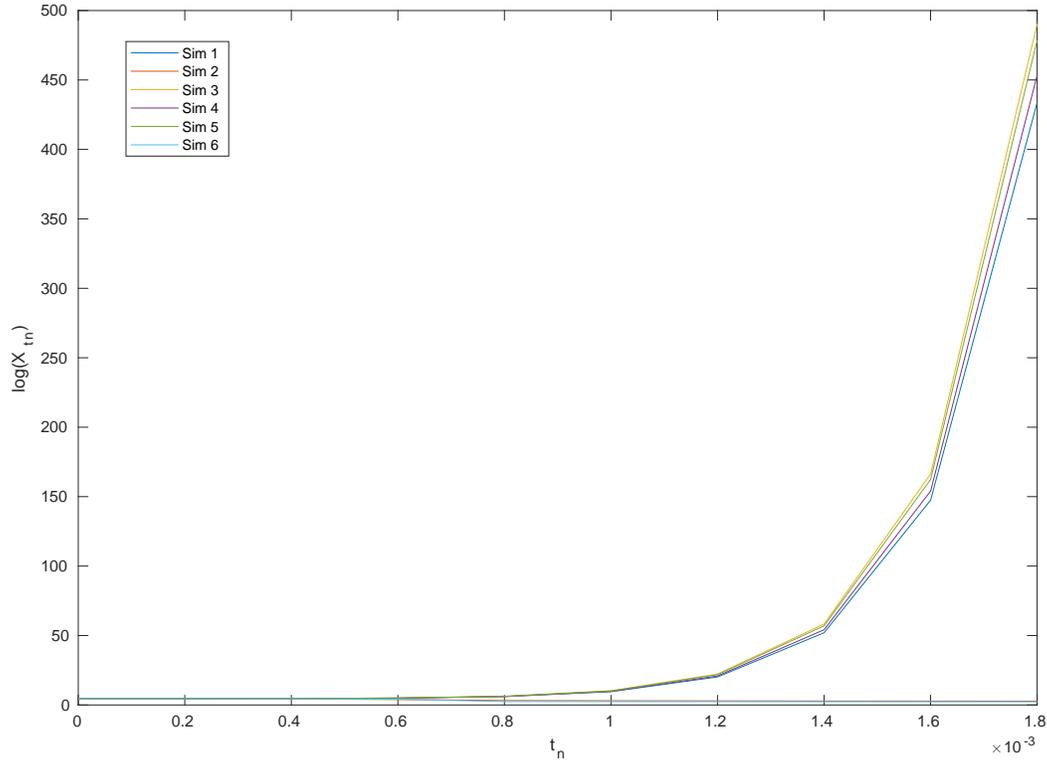}
\end{figure}

\textbf{Note:} From Lemma \ref{lemma-counter} we know that as $\Delta$ decreases, the probability of explosion decreases. Thus, for ``very small" $\Delta$ (say less than $10^{-4}$) we couldn't find one explosion in 100,000 simulations. 

In addition, we simulated the adaptive-EM solution of the SDDE \eqref{SDDE counterexample} using the function $h^\delta$ defined in \eqref{hs-counter}. As we proved in Section \ref{stab SDDEs}, the solution is a.s. exponentially stable. Figure \ref{adaptive} shows 10,000 paths of the adaptive-EM solution.

\begin{figure} [H]
	\caption{Simulations of adaptive-EM solution  \label{adaptive}}
	\includegraphics[scale=0.6]{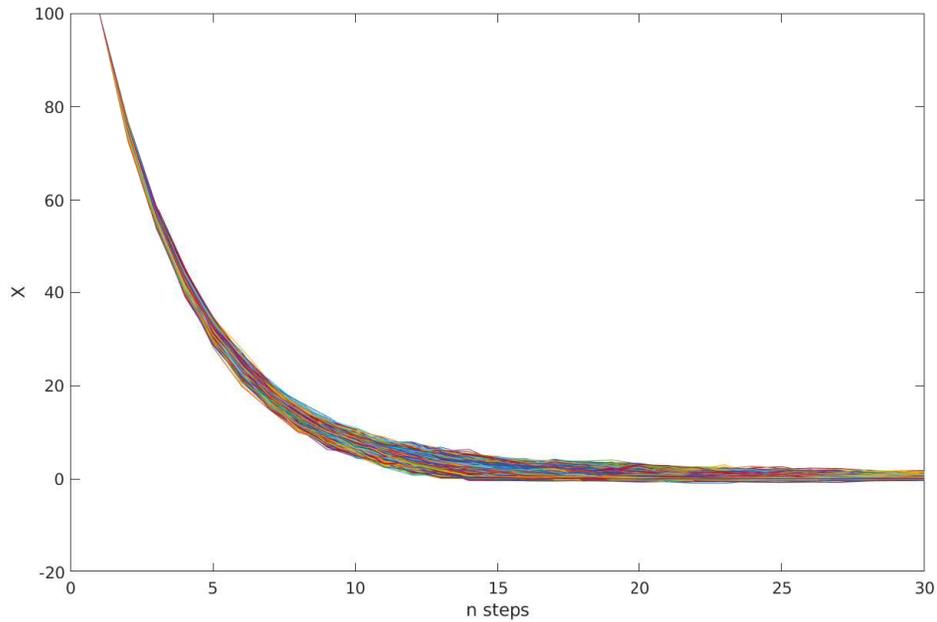}
\end{figure}

The next graph shows the first 10 values of $h^{\delta}(\hat X_{t_n})$ for two different simulations. At the start, $\hat X_0 =100,$ so the term $-\hat X^3_{t_n}$ dominates the equation, making the diffusion term very ``big'' (in absolute value) in comparison with $\hat X_{t_n}.$ Therefore, the adaptive step is very ``small" at the beginning and increases progressively as the ratio $f(\hat X_{t_n},\hat X_{\hat t_n})/\hat X_{t_n}$ decreases. This ensures all the simulated paths to decay exponentially in a ``small" number of steps.  
 
 \begin{figure} [H]
 	\caption{The first ten adaptive steps for two different simulations  \label{hs}}
 	\includegraphics[scale=0.8]{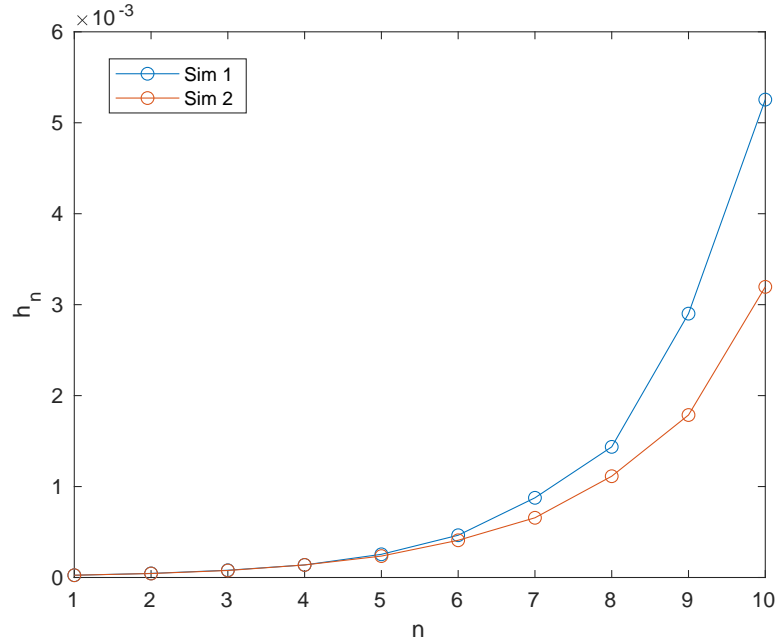}
 \end{figure}

\end{document}